\documentclass{article} 

\usepackage{amssymb}
\usepackage{amsthm}
\usepackage{amsmath}
\usepackage{enumerate}
\usepackage{mathtools}
\usepackage{blindtext}
\usepackage{enumitem}
\usepackage{bm}
\mathtoolsset{showonlyrefs=true} 
\usepackage{booktabs}
\newcommand{\ra}[1]{\renewcommand{\arraystretch}{#1}}
\usepackage{float}
\usepackage{graphicx}
\usepackage{adjustbox}

\usepackage{algorithm}

\usepackage[pdftex,
plainpages=false,
bookmarks,
bookmarksnumbered,
colorlinks=true,
linkcolor=blue,
citecolor=blue,
urlcolor=blue,
filecolor=black
]{hyperref}
\usepackage{hyperref}

\usepackage{fancybox,color}
\definecolor{lred}{rgb}{1,0.8,0.5}
\definecolor{lblue}{rgb}{0.8,0.8,1}
\definecolor{dred}{rgb}{0.6,0,0}
\definecolor{dblue}{rgb}{0,0,0.7}
\definecolor{violet}{rgb}{0.5804,0.0000,0.8275}
\definecolor{purple}{rgb}{0.2400,0.5700,0.2500}
\definecolor{TGreen}{rgb}{0,0.50,0.10}
\def\@themcountersep{}
\newtheorem{lemma}{Lemma}[section]

\newtheorem{assumption}[lemma]{Assumption}
\newtheorem{theorem}[lemma]{Theorem}
\newtheorem{corollary}[lemma]{Corollary}

\newtheorem{remark}[lemma]{Remark}

\def\0{\mbox{\bf 0}}
\def\1{\mbox{\bf 1}}
\def\2{\mbox{\bf 2}}
\def\3{\mbox{\bf 3}}
\def\4{\mbox{\bf 4}}
\def\5{\mbox{\bf 5}}
\def\6{\mbox{\bf 6}}
\def\7{\mbox{\bf 7}}
\def\8{\mbox{\bf 8}}
\def\9{\mbox{\bf 9}}

\def\b{\mbox{\boldmath $b$}}

\newdimen\zhige \zhige=0pt
\def\chige#1{{\setbox\zhige\hbox{#1}\ifdim\ht\zhige=1ex\accent24 #1%
  \else\ooalign{\unhbox\zhige\crcr\hidewidth\char24\hidewidth}\fi}}

\def\BSigma{\mbox{\boldmath $\Sigma$}}

\def\g{\mbox{\boldmath $g$}}

\def\s{\mbox{\boldmath $s$}}

\def\u{\mbox{\boldmath $u$}}

\def\y{\mbox{\boldmath $y$}}

\def\R{\mathbb{R}}
\def\A{\mathcal{A}}
\def\B{\mathcal{B}}
\def\F{\mathcal{F}}
\def\M{\mathcal{M}}
\def\N{\mathcal{N}}
\def\W{\mathcal{W}}
\def\Q{\mathcal{Q}}
\def\S{\mathcal{S}}
\def\Z{\mathcal{Z}}

\def\bS{\mathbb{S}}

\def\C{\mbox{\boldmath $C$}}
\def\D{\mbox{\boldmath $D$}}
\def\E{\mbox{\boldmath $E$}}

\def\I{\mbox{\boldmath $I$}}

\def\O{\mbox{\boldmath $O$}}

\def\T{\mbox{\boldmath $T$}}
\def\U{\mbox{\boldmath $U$}}
\def\V{\mbox{\boldmath $V$}}
\def\X{\mbox{\boldmath $X$}}

\def\Z{\mbox{\boldmath $Z$}}
\def\AC{\mbox{$\cal A$}}

\def\FC{\mbox{$\cal F$}}

\def\LC{\mbox{$\cal L$}}
\def\MC{\mbox{$\cal M$}}
\def\NC{\mbox{$\cal N$}}
\def\OC{\mbox{$\cal O$}}

\def\SC{\mbox{$\cal S$}}

\def\WC{\mbox{$\cal W$}}

\def\Real{\mbox{$\mathbb{R}$}}

\def\SMAT{\mbox{$\mathbb{S}$}}

\def\pib{\mbox{\boldmath $\pi$}}

\usepackage[normalem]{ulem}
\ifx11 

\fi
\title{A new dual spectral projected gradient method 
for log-determinant semidefinite programming 
with hidden clustering structures}

\author{Charles Namchaisiri\thanks{School of Computing, Tokyo Institute of Technology,  Japan. (\href{namchaisiri.c.aa@m.titech.ac.jp}{namchaisiri.c.aa@m.titech.ac.jp})} \and
Tianxiang Liu\thanks{School of Computing, Tokyo Institute of Technology,  Japan. (\href{liu@c.titech.ac.jp}{liu@c.titech.ac.jp})} 
\and Makoto Yamashita\thanks{School of Computing, Tokyo Institute of Technology,  Japan. (\href{Makoto.Yamashita@c.titech.ac.jp}{Makoto.Yamashita@c.titech.ac.jp})}  }
\numberwithin{equation}{section}
\allowdisplaybreaks

\begin{document}
\maketitle

\begin{abstract}

In this paper, we propose a new efficient method for a sparse Gaussian graphical model with hidden clustering structures by extending a dual spectral projected gradient (DSPG) method proposed by Nakagaki et al.~(2020). 
We establish the global convergence of the proposed method to an optimal solution,
and we show that the projection onto the feasible region can be solved with a low computational complexity by the use of the pool-adjacent-violators algorithm.  
Numerical experiments on synthetic data and real data demonstrate the efficiency of the proposed method. The proposed method takes 0.91 seconds to achieve a similar solution to the direct application of the DSPG method which takes 4361 seconds.
\end{abstract}

\section{Introduction}\label{sec:introduction}
In this paper, we address the following optimization problem:
\begin{equation}\label{primal}
\begin{split}
 \min_{\bm X\in\bS^n} & \ \ f(\bm X) := \bm C \bullet \bm X - \mu \log \det \bm X + \rho \sum_{i < j}|X_{ij}| + \lambda \sum_{i < j}\sum_{s < t}|X_{ij} - X_{st}|\\
\mbox{s.t.}  & \ \  \A (\bm X) = \bm b, \, \bm X \succ \bm 0,
\end{split}
\end{equation}
where $\mu > 0$, $\rho > 0$, $\lambda > 0$, $\bm C\in\bS^n$ and $\bm b\in\R^m$ are given,  and $\A: \bS^n \to \R^m$ is a linear map defined by $\A(\cdot) := (\bm A_1 \bullet \cdot,\, \bm A_2 \bullet \cdot ,\,\ldots, \bm A_m \bullet \cdot)^\top$ with given matrices $\bm A_1,\,\bm A_2,\,\dots,\bm A_m \in \bS^n$. 
The model \eqref{primal} was introduced by Lin et al.~\cite{lin2020estimation} to estimate sparse Gaussian graphical models with hidden clustering structures. The third and the fourth terms are introduced for inducing the sparsity in $\bm X$ and the clustering structure of the concentration matrix, respectively.

To describe the structure of the fourth term, define the linear map $vect: \SMAT^{n} \to \Real^{\frac{n(n-1)}{2}}$, which is the map that converts the strictly upper triangular part of the symmetric matrix sequentially to the vector of dimension $n(n-1)/2$. We define the mapping $\mathcal{K}: \{1,2,\dots,n\} \times \{1,2,\dots,n\} \to \{1,2,\dots,n(n-1)/2\}$ by $\mathcal{K}(i,j)$ is the position of the $(i,j)$th component of matrix $\bm X$ in $vect(\bm X)$. Since this mapping is bijective, we can also define the inverse mapping $\mathcal{K}^*: \{1,2,\dots,n(n-1)/2\} \to \{1,2,\dots,n\} \times \{1,2,\dots,n\}$ by $\mathcal{K}^*(a)$ is the position of the $a$th component of vector $vect(\bm X)$ in $\bm X$. We introduce a linear map 
$\Q : \Real^{\frac{n(n-1)}{2}} \to \SMAT^{\frac{n(n-1)}{2}}$ defined by
\begin{align}
[\Q(\bm X)]_{a,b} = \left\{
        \begin{array}{ll}
            X_{\mathcal{K}^*(a)} - X_{\mathcal{K}^*(b)} & \text{if} \  a < b \\
            {[\Q(\bm X)]_{b,a}} & \text{if} \ a > b \\
            0 & \text{if} \  a = b,
        \end{array}
        \right.
\end{align}
then the fourth term can be expressed with $\lambda \|\Q(\bm X)\|_1$. 

Lin et al.~\cite{lin2020estimation} reported that the estimated sparsity pattern obtained from this model is close to the true sparsity pattern, which is useful for recovering the graph structure.
They proposed a two-phase algorithm to solve \eqref{primal}. 

On the other hand, the model \eqref{primal} without the last (fourth) term 
\begin{equation}\label{primal-without-fourth}
	\begin{split}
	 \min_{\bm X\in\bS^n} & \ \ f_0(\bm X) := \bm C \bullet \bm X - \mu \log \det \bm X + \rho \sum_{i \ne j}|X_{ij}| \\
	\mbox{s.t.}  & \ \  \A (\bm X) = \bm b, \, \bm X \succ \bm 0,
	\end{split}
	\end{equation}
was well studied over the years and many solution methods have been proposed; see \cite{nakagaki2020dual,Wang2016,WST2010,YST2013}. In particular, Nakagaki et al.~\cite{nakagaki2020dual} proposed a dual spectral projected gradient (DSPG) method. The DSPG method solves the corresponding dual problem of \eqref{primal-without-fourth} below:
 \begin{equation}\label{dual_dspg}
\begin{split}
 \max_{\bm y\in\R^m,\,\bm W\in\bS^n} & \ \  \bm b^\top\bm y   + \mu\log\det\left(\bm C - \A^\top(\bm y) + \bm W \right) + n \mu - n \mu \log \mu \\
\mbox{s.t.}  & \ \ \|\bm W\|_{\infty} \le \rho\\
& \ \ \bm C - \A^\top (\bm y) + \bm W \succ\bm 0.
\end{split}
\end{equation}
We use $\|\bm W\|_{\infty} := \max \{ |W_{ij}| \}$ to denote the maximum absolute element of $\bm W$, and 
$\A^\top : \R^m \to \bS^n$  the adjoint operator of $\A$ (that is, $\A^\top (\y) := \sum_{i=1}^m \bm A_i y_i$).
The DSPG method is an iterative method; in each iteration, 
it evaluates the gradient of the objective function 
of \eqref{dual_dspg}, computes the projection onto 
$\| \bm W \|_{\infty} \le \rho$, and calculates  
the step length to generate the next iteration point
in the interior of the region of 
$\bm C - \A^\top (\bm y) + \bm W \succ\bm 0$.

By reformulating the fourth term in the objective function of \eqref{primal} similarly, the DSPG method is directly applicable to \eqref{primal}. The dual problem of \eqref{primal} is
\begin{equation}\label{dual_pre}
\begin{split}
 \max_{\bm y\in\R^m,\,\bm W\in\bS^n,\,\bm Z\in\bS^{\frac{n(n-1)}{2}}} & \ \  \bm b^\top\bm y   + \mu\log\det\left(\bm C - \A^\top(\bm y) + \frac{\bm W}{2} + \Q^\top (\bm Z)\right) + n \mu - n \mu \log \mu \\
\mbox{s.t.}  & \ \ \|\bm W\|_{\infty} \le \rho ,\, \|\bm Z\|_{\infty} \le \lambda,\\
& \ \ \bm C - \A^\top (\bm y) + \frac{\bm W}{2} + \Q^\top (\bm Z) \succ\bm 0.
\end{split}
\end{equation}
Here, $\bm b^\top$ is the transpose of the vector $\bm b$  and $\Q^\top : \bS^{\frac{n(n-1)}{2}} \to \bS^n$ the adjoint operator of $\Q$, thus,
\begin{align}
        [\Q^\top (\Z)]_{i,j} = \left\{
        \begin{array}{ll}
            \sum_{l > \mathcal{K}(i,j)} Z_{\mathcal{K}(i,j),l}
            - \sum_{l < \mathcal{K}(i,j)} Z_{l,\mathcal{K}(i,j)} & \text{if} \  i < j \\
            {[\Q^\top (\Z)]_{ji}} & \text{if} \ i > j \\
            0 & \text{if} \  i = j.
        \end{array}
        \right.
\end{align}

If we group two variables $\bm W$ and $\bm Z$ in \eqref{dual_pre}, then the DSPG method is applicable. However, such a direct application of the DSPG method requires the gradient of the objective function with respect to $\bm W$ and $\bm Z$.
In particular, the evaluation of the gradient for $\bm Z$ 
amounts to $O(n^4)$ operators, which is expensive even for a moderate size of $n$. To resolve this issue, it is necessary to reduce the computational cost by exploiting the structure of $\Q^\top$.

In this paper, we propose a new efficient method for solving \eqref{primal} by extending the DSPG method. 
Specifically, 
to avoid the high computational cost due to $\Q^\top$,
we first reformulate \eqref{dual_pre}, then apply a modified DSPG method to the reformulated model. In the reformulated model, we move the difficult structure of $\Q^\top$ from the objective function into the constraints, but this brings an extra cost for computing the projection onto the new constraint set. We demonstrate that such a subproblem can be equivalently written as the projection problem onto the ordered constraint set and this projection can be efficiently computed with the pool-adjacent-violators algorithm (PAVA)~\cite{PAVApaper,henzi2020accelerating}.

In this paper, our main contributions are as follows.
\begin{itemize}
\item We propose a new method (Algorithm~\ref{alg}) for solving \eqref{primal}, and establish the convergence analysis of the method. 

\item We discuss how to solve the subproblem of Algorithm~\ref{alg} efficiently and show that the computational complexity of the subproblem is significantly decreased compared with that of the direct application of the DSPG method to \eqref{dual_pre}.

\item We illustrate the  efficiency of Algorithm~\ref{alg} by numerical experiments (Section~\ref{numerical}). We obtain 
a notable reduction in the computation time from the DSPG method.
\end{itemize}

The remainder of this paper is organized as follows. 
We describe our DSPG-based method in Section~\ref{sec:method} and establish its convergence analysis in Section~\ref{analysis}. In Section~\ref{subproblem}, we discuss how to solve the subproblem efficiently and evaluate the computational complexity. In Section~\ref{numerical}, we show numerical results to verify the efficiency of our method.
Finally, we conclude this paper and discuss future directions in  Section~\ref{sec:conclusion}.

\subsection{Notation}\label{sec:notation}

Let $\bS^n$ be the set of symmetric matrices of dimension $n$.
The inner product between two matrices $\bm A, \bm X \in \bS^n$ is 
defined by $\bm A \bullet \bm X := \sum_{i=1}^n \sum_{j=1}^n A_{ij} X_{ij}$. We use the notation $\bm X \succ (\succeq) \ \bm 0$ to denote that 
$\bm X \in \bS^n$ is positive definite (positive semidefinite, respectively).

Given $\bm x\in\R^n$, 
let $\|\bm x\|:= \sqrt{\bm x^\top \bm x}$ denote the Euclidean norm.
 Given $\bm X\in\bS^n$, we let $\|\bm X\|: = 
 \sqrt{\bm X \bullet \bm X}$ denote its 
 Frobenius norm.
In the direct product space $\R^m\times\bS^n\times\bS^{\ell}$, the inner product for $\bm U_1 = (\bm y_1,\,\bm W_1,\,\bm S_1)\in\R^m\times\bS^n\times\bS^{\ell}$ and $\bm U_2 = (\bm y_2,\,\bm W_2,\,\bm S_2)\in \R^m\times\bS^n\times\bS^{\ell}$ is defined as 
 \begin{equation*}
 \langle\bm U_1,\,\bm U_2\rangle := \bm y_1^\top\bm y_2 + \bm W_1\bullet\bm W_2 + \bm S_1\bullet\bm S_2.
 \end{equation*}
We also define the norm of $\bm U \in \R^m\times\bS^n\times\bS^{\ell}$ as $||\bm U||:= \sqrt{\langle\bm U,\,\bm U\rangle}$.

Given a linear map $\A$, its adjoint operator is written as $\A^\top$. 
We define the operator norm of $\A :  \SMAT^n \to \Real^m$ 
as $\|\A\| := \sup_{\bm X \ne \0} \left\{\frac{\|\A(\bm X)\|}{\|\bm X\|}\right\}$, and the operator norm of $\A^\top : \Real^m \to \SMAT^n$ as $\|\A^\top\| := \sup_{\bm y \ne \0} \left\{\frac{\|\A^\top(\bm y)\|}{\|\bm y\|}\right\}$. 
For a closed convex set $\Omega$, we use $P_{\Omega}(\cdot)$ to denote the projection onto $\Omega$, i.e.,
\begin{equation*}
P_{\Omega}(\cdot) := \arg\min_{\bm x\in\Omega}\|\bm x - \cdot\|.
\end{equation*}
For simplicity of notation, let 
\begin{equation*}
\W:= \{\bm W \in \SMAT^n \mid \|\bm W\|_{\infty} \le \rho\}.
\end{equation*}

\section{A new DSPG-based method}\label{sec:method}

Nakagaki et al.~\cite{nakagaki2020dual} proposed a dual spectral projected gradient (DSPG) method to solve \eqref{dual_dspg} that does not involve $\bm Z$. 
A key idea of the DSPG method is to combine a non-monotone line search projected gradient method with a feasible step adjustment. As stated in Section~\ref{sec:introduction}, if we directly apply the DSPG method to solve \eqref{dual_pre}, we need to calculate the gradient of the objective function in \eqref{dual_pre} in each iterate, which is expensive due to the structure of $\bm Z$. 

To resolve this difficulty, we introduce a new variable that represents $\Q^\top (\bm Z)$. Specifically, we introduce a new set:
\begin{equation}
\S:= \left\{ \bm S \in \bS^n \ \mid  \bm S = \Q^\top (\bm Z), \  \|\bm Z\|_{\infty} \le \lambda\right\}.\label{S-definition}
\end{equation}
 Using this set and $\W$ and combining the variables $(\bm y, \bm W, \bm S)$ into one composite variable $\bm U$, we can rewrite \eqref{dual_pre} as the following optimization problem:
\begin{equation}\label{dual}
\begin{split}
 \max_{\bm U:= \left(\bm y,\, \bm W,\, \bm S\right)\in\R^m\times\bS^n\times\bS^n} & \ \  g(\bm U):= \bm b^\top\bm y   + \mu\log\det\left(\bm C - \A^\top (\bm y) + \frac{\bm W}{2} + \bm S\right) + n \mu - n \mu \log \mu \\
\mbox{s.t.}  & \ \ \bm W\in\W,\, \bm S\in\S, \ 
 \bm C - \A^\top (\bm y) + \frac{\bm W}{2} + \bm S \succ\bm 0.
\end{split}
\end{equation}

We list the blanket assumption which also appears in \cite{nakagaki2020dual} as follows:
\begin{assumption}\label{assp}
We assume the following statements hold for problem \eqref{primal}
and its corresponding dual \eqref{dual}:
\begin{itemize}
\item[{\rm (i)}] The linear map $\A$ is surjective;
\item[{\rm (ii)}] The primal problem~\eqref{primal} has a strictly feasible solution $\widehat{\bm X} \succ \bm 0$ such that $\A (\widehat{\bm X}) = \bm b$;
\item[{\rm (iii)}] The dual problem~\eqref{dual} has a feasible solution. 
\end{itemize}
\end{assumption}

With the introduction of a new variable $\bm S$ in \eqref{dual}, the computational cost for the gradient is now reduced. The computational cost of $\nabla_{\S} g$ is $O(n^2)$,
which is significantly reduced from the $O(n^4)$ cost 
of the gradient of the objective function in \eqref{dual_pre}
with respect to $\bm Z$, since 
the numbers of elements in $\bm S$ and $\bm Z$ are 
$O(n^2)$ and $O(n^4)$, respectively.
On the other hand, the difficulty is now embedded into the computation of the projection onto set $\S$ in \eqref{S-definition}, which we will discuss about this problem in Section~\ref{subproblem}. 

Let $\F$ denote the feasible set of \eqref{dual}. Then we can write it as $\F = \MC \cap \NC$, where
\begin{align}
\M &:= \mathbb{R}^m \times \mathcal{W} \times \mathcal{S}, \\
\N &:= \left\{(\bm y,\bm W,\bm S) \in \R^m \times \bS^n \times \bS^{n} : \bm C  - \A^\top (\bm y) + \frac{\bm W}{2} + \bm S \succ \bm 0\right\}.
\end{align}
In the proximal gradient step, we compute $\nabla g$ in \eqref{dual} and the projection onto $\M$. Note that $\W$ is a simple set whose projection operator yields a closed-form solution. In contrast, the projection operator of set $\S$ is more complicated; it will be discussed in Section~\ref{subproblem} and it is 
the main difference from Nakagaki et al.~\cite{nakagaki2020dual}.

We present the framework of our new DSPG-based algorithm as Algorithm~\ref{alg}. For the simplicity of notation in the algorithm, we introduce a linear map
\begin{equation*}
\B(\bm U) := -\A^\top (\bm y) + \frac{\bm W}{2} + \bm S, \ \ \mbox{where}\ \   \bm U:= \left(\bm y,\, \bm W,\, \bm S\right)\in\R^m\times\bS^n\times\bS^n,
\end{equation*}
which allows us to rewrite the set $\N$ simply as $\N = \left\{\bm U\in\R^m\times\bS^n\times\bS^n: \bm C + \B(\bm U) \succ 0\right\}$. 
We let $\bm U^k:= \left(\bm y^k,\, \bm W^k,\, \bm S^k\right)$. 

\begin{algorithm}[H]
\caption{A new DSPG-based algorithm for solving \eqref{dual}}
\label{alg}
\begin{description}

\item [Initialization.] Choose  parameters $\varepsilon > 0,\, \tau \in (0,1),\, \gamma \in (0,1), \, 0 <\beta< 1,\, 0 < \alpha_{\min} < \alpha_{\max} < \infty$ and  integer  $M > 0$. Take  $\bm U^0 \in \F$ and $\alpha_0 \in [\alpha_{\min} ,\, \alpha_{\max}]$. Set $k = 0$.

\item [Step 1.] Let $\bm R^k := \left(\Delta\bm y^k_{(1)},\, \Delta\bm W^k_{(1)},\, \Delta\bm S^k_{(1)} \right) = P_{\M}\left(\bm U^k + \nabla g(\bm U^k)\right)  - \bm U^k$. If $\|\bm R^k\| \le \varepsilon$, terminate;  otherwise, go to \textbf{Step 2}.

\item [Step 2.]   Let $\bm D^k:= \left(\Delta\bm y^k,\, \Delta\bm W^k,\, \Delta\bm S^k \right) = P_{\M}\left(\bm U^k + \alpha_k\nabla g(\bm U^k)\right) - \bm U^k$. Let $\bm L_k\bm L_k^\top:= \bm C + \B(\bm U^k)$ be the Cholesky decomposition and $\theta$ be the minimum eigenvalue of $\bm L_k^{-1}\B(\bm D^k) \left(\bm  L_k^\top\right)^{-1}$. Set
\begin{equation*}
 \nu_k :=
\begin{cases}
  1 & \text{if} \ \theta \ge 0,\\
 \min\{1,\, -\tau/\theta \}  & \text{otherwise}.
\end{cases}
\end{equation*}
Apply a line search to find the largest element $\sigma_k \in\{1,\,\beta,\,\beta^2, \ldots\}$ such that
 \begin{equation}
g(\bm U^k + \sigma_k \nu_k\bm D^k) \geq  \min\limits_{[k - M + 1]_+ \le j \le k} g(\bm U^j) + \gamma\sigma_k \nu_k\langle\nabla g(\bm U^k),\,\bm D^k\rangle.
\end{equation}



\item [Step 3.]  Let $\bm U^{k+1} = \bm U^k + \sigma_k \nu_k\bm D^k$. Let $p_k:= \langle\bm U^{k+1} - \bm U^k,\, \nabla g(\bm U^{k+1}) - \nabla g(\bm U^k)\rangle$. Set
\begin{equation*}
 \alpha_{k+1} :=
\begin{cases}
  \alpha_{\max} & \text{if} \ p_k \ge 0,\\
\min\left\{\alpha_{\max},\,\max\left\{\alpha_{\min},\, -\|\bm U^{k+1} - \bm U^k\|^2/p_k\right\}\right\} & \text{otherwise}.
\end{cases}
\end{equation*}
Set $k \leftarrow k + 1$. Return to \textbf{Step 1}.

\end{description}

\end{algorithm}

Note that the projection $P_{\M}(\bm U)$ can be decomposed as $(\y , P_{\WC}(\bm W) , P_{\SC}(\bm S))$, since the projections onto $\R, \mathcal{W}, \mathcal{S}$ are independent.
In addition, the DSPG method~\cite{nakagaki2020dual} and Algorithm~\ref{alg}
are non-monotone gradient methods,
as Step 2 guarantees a non-monotone increase in the objective
function (See~\cite{birgin2000nonmonotone}).

\begin{remark}[Remark on the computation in Algorithm~\ref{alg}]
	\label{REMA:X-U}
Algorithm~\ref{alg} involves the computation of projection onto $\M:=\R^m\times\W\times\S$. While the projection onto $\W$ has a closed-form solution, the projection onto $\S$ is not easy, which will be discussed in Section~\ref{subproblem}. Another computational cost lies in the computation of $\nabla g(\bm U^k)$. Denote
\begin{equation}
\bm X(\bm U) := \mu \left(\bm C + \B(\bm U)\right)^{-1}.
\end{equation}
Then we can express 
\begin{equation*}
\nabla g(\bm U)  = \Big(\bm b - \A^\top\left(\bm X(\bm U)\right),\, \frac{1}{2}\bm X(\bm U),\, \bm X(\bm U)\Big).
\end{equation*}

\end{remark}

\section{Convergence Analysis}\label{analysis}

In this section, we present the convergence properties of Algorithm~\ref{alg}. We start with the following lemma stating the boundedness of the sequence.

\begin{lemma}[Boundedness of sequence]\label{LEMM:bounded}
Let sequence $\{\bm U^k\}$ be generated by Algorithm~\ref{alg}. Define a level set
\begin{equation*}
\mathcal{L} :=  \left\{\bm U \in \R^m \times \mathbb{S}^n \times \mathbb{S}^n : \bm U \in \F,  \ g(\bm U) \ge g(\bm U^0)\right\}.
\end{equation*}
Then, $\{\bm U^k\}\subseteq\mathcal{L}$ and $\{\bm U^k\}$ is bounded. 
\end{lemma}

\begin{proof}
We first prove that $\{\bm U^k\}\subseteq\mathcal{L}$ by induction. It is easy to see that $\bm U^0\in\mathcal{L}$. We assume that $\bm U^k\in\mathcal{L}$ holds for some $k \ge 0$. Since $\sigma_k\nu_k\in(0,\,1]$, we have that $\bm U^{k+1}$ is a convex combination of $P_{\mathcal{M}}\left(\bm U^k + \alpha_k \nabla g(\bm U^k)\right)$ and $\bm U^k$, which together with $\bm U^k\in\mathcal{M}$ and the convexity of $\mathcal{M}$ implies that $\bm U^{k+1}\in\mathcal{M}$. Moreover,  we see from the update of $\nu_k$ that
\begin{equation*}
\begin{split}
& \bm L_k^{-1}\big(\bm C + \B(\bm U^{k+1})\big)\left(\bm  L_k^\top\right)^{-1} =\bm L_k^{-1}\big(\bm C + \B\big(\bm U^k + \sigma_k\nu_k\bm D^k\big)\big)\left(\bm  L_k^\top\right)^{-1} \\
= &  \bm L_k^{-1}\big(\bm C + \B(\bm U^k)\big)\left(\bm  L_k^\top\right)^{-1} + \sigma_k\nu_k \bm L_k^{-1}\B(\bm D^k)\left(\bm  L_k^\top\right)^{-1} = \bm I + \sigma_k\nu_k \bm L_k^{-1}\B(\bm D^k)\left(\bm  L_k^\top\right)^{-1} \succeq \bm I + \sigma_k\nu_k\theta\bm I \succ \bm 0.
\end{split}
\end{equation*}
This together with $\bm U^k\in\N$ implies that $\bm U^{k+1}\in\N$, therefore, it holds that $\bm U^{k+1}\in\F$. By induction, this proves that $\bm U^k\in\F$ for all $k$.

On the other hand, since $\M$ is convex and $\bm U^k\in\M$, using Proposition 2.1~(1) in \cite{HAGER08}, we have for all $k$ that 
\begin{equation*}
\begin{split}
0 & \ge \left\langle P_{\M}\left(\bm U^k + \alpha_k\nabla g(\bm U^k)\right) - \left(\bm U^k + \alpha_k\nabla g(\bm U^k)\right) ,\, P_{\M}\left(\bm U^k + \alpha_k\nabla g(\bm U^k)\right) - \bm U^k\right\rangle \\
 & = \left\langle\bm D^k -\alpha_k \nabla g(\bm U^k) ,\, \bm D^k\right\rangle = \|\bm D^k\|^2 - \alpha_k\left\langle\nabla g(\bm U^k),\,\bm D^k\right\rangle.
\end{split}
\end{equation*}
This implies that $\left\langle\nabla g(\bm U^k),\,\bm D^k\right\rangle\ge \frac{\|\bm D^k\|^2}{\alpha_k} \ge \frac{\|\bm D^k\|^2}{\alpha_{\max}} \ge 0$, which together with the line search in Step 2 in Algorithm~\ref{alg} proves that $g(\bm U^k)\ge g(\bm U^0)$. Consequently, we have $\{\bm U^k\}\subseteq\mathcal{L}$.

It then shows that $\{\bm U^k\}$ is bounded. 
Note that $\bm U^k\in\F = \M\cap\N$ and $\M = \R^n\times\W\times\S$, where $\W$ is bounded, $\S$ is the image of $\Q$ on a bounded set. Therefore, $\{\bm W^k\}$ and $\{\bm S^k\}$ are bounded, thus showing the boundedness of $\{\bm y^k\}$ is enough. 
By $\{\bm U^k\}\subseteq\mathcal{L}$ and Assumption~\ref{assp}~(i), we have 
\begin{equation*}
\begin{split}
g(\bm U^0) \le g(\bm U^k) & = \bm b^\top\bm y^k + \mu\,{\rm log}\,{\rm det}\left(\bm C + \B (\bm U^k) \right)\\
& = \langle\widehat{\bm X},\, \A^\top(\bm y^k)\rangle + \mu\,{\rm log}\,{\rm det}\left(\bm C + \B (\bm U^k) \right) \\
& = \langle\widehat{\bm X},\, \bm C + \frac{\bm W^k}{2} + \bm S^k \rangle  - \langle\widehat{\bm X},\, \bm C + \B (\bm U^k) \rangle + \mu\,{\rm log}\,{\rm det}(\bm C + \B (\bm U^k)). 
\end{split}
\end{equation*}
Since $\widehat{\bm X}\succ\bm 0$, and $\{\bm W^k\}$ and $\{\bm S^k\}$ are bounded, we then have $\{\bm C + \B (\bm U^k)\}$ is bounded. 
This together with Assumption~\ref{assp}~(ii) ($\A$ is surjective) indicates that $\{\bm y^k\}$ is bounded. This completes the proof.
\end{proof}

\begin{theorem}[Optimality condition]\label{THEO:optimality}
$\bm U^*$ is an optimal solution of \eqref{dual} if and only if $\bm U^*\in\F$ and there exists some $\alpha > 0$ such that
\begin{equation*}
P_{\M}\left(\bm U^* + \alpha\nabla g(\bm U^*)\right) = \bm U^*.
\end{equation*}
\end{theorem}

\begin{proof}
Note that we can rewrite \eqref{dual} as
\begin{equation}
\max_{\bm U}\ \ g(\bm U) - \delta_{\F}(\bm U), \label{eq:max-g-delta-F}
\end{equation}
where $\delta_{\F}$ is the indicator function with respect to $\FC$.
Let $N_{\F}(\bm U)$ be the normal cone at $\bm U$.
Since the objective function in \eqref{eq:max-g-delta-F} is a convex problem, $\bm U^*$ is an optimal solution if and only if $\bm U^*\in\F$ and
\begin{equation*}
\bm 0 \in \nabla g(\bm  U^*) - N_{\F}(\bm U^*) = \nabla g(\bm  U^*) - N_{\M\cap\N}(\bm U^*) = \nabla g(\bm  U^*) - N_{\M}(\bm U^*),
\end{equation*}
where the last equality follows from \cite[Theorem~3.30]{cui} and that $\N$ is an open set. This can be rewritten as $\nabla g(\bm  U^*)\in N_{\M}(\bm U^*)$, which is further equivalent to $\bm U^*\in\F$ and the existence of $\alpha > 0$ such that 
\begin{equation*}
\langle\bm U - \bm U^*,\, \alpha\nabla g(\bm U^*)\rangle \le 0, \ \ \ \forall\, \bm U\in\M.
\end{equation*}
We rewrite the above as $\bm U^*\in\F$ and
\begin{equation*}
\langle\bm U - \bm U^*,\, \left(\bm U^* + \alpha\nabla g(\bm U^*)\right) - \bm U^*\rangle \le 0, \ \ \ \forall\, \bm U\in\M.
\end{equation*}
Due to the uniqueness of the projection onto the convex set $\M$ and Proposition 2.1~(1) in \cite{HAGER08}, we see that this is equivalent to $\bm U^*\in\F$ and the existence of $\alpha > 0$ such that
\begin{equation*}
P_{\M}\left(\bm U^* + \alpha\nabla g(\bm U^*)\right) = \bm U^*.
\end{equation*}
This completes the proof.
\end{proof}
The following corollary can be derived from 
Lemma~\ref{LEMM:bounded}.
\begin{corollary}\label{CORO:bound-on-X}
	A set $\left\{ \X(\bm U) \mid \bm U \in \LC \right\}$ is bounded. 	
\end{corollary}
We define the lower bound and upper bound of $\|\X(\bm U)\|$ by $\beta_{\min}$ and $\beta_{\max}$, respectively. In other words, $\beta_{\min}\bm I \preceq \X(\bm U) \preceq \beta_{\max} \bm I$. Letting $\bm X^k : = \bm X(\bm U^k)$, we consider the boundedness of components.
\begin{lemma}\label{LEMM:bound}
	There exist positive constants $\eta_{\bm X}, \eta_{\bm X^{-1}}, \eta_{ \Delta\bm y},
\eta_{\Delta \bm W}$ and $\eta_{\Delta \bm S}$ such that 
 $\|\X^k\| \le \eta_{\bm X}$,\\ $\|(\X^k)^{-1}\| \le \eta_{\bm X^{-1}}$, $\| \Delta \y^k \| \le \eta_{\Delta \bm y}$, $\| \Delta \bm W^k\| \le \eta_{\Delta \bm W}$, and $ \| \Delta \bm S^k \| \le \eta_{\Delta \bm S}$ hold for all $k$. 
\end{lemma}

\begin{proof}
	By following Remark~2 in \cite{nakagaki2020dual}, we obtain $\eta_{\X}$ and $\eta_{\X^{-1}}$.
	In addition, it holds that $\| \Delta \y^k \| \leq \alpha_k (\| \b\| + \|\AC\| \cdot \|\X^k\|) \leq \alpha_{\max} (\| \b \| + \eta_{\X} \| \AC \|) =: \eta_{\Delta \y}$.
   From the inequalities $\left\| P_{\W}\left(\bm W^k + \frac{\alpha_k}{2} \bm X^k\right) - \bm W^k \right\| \leq \| \frac{\alpha_k}{2} \bm X^k\| $ and $\left\| P_{\SC}\left(\bm S^k + \alpha_k \X^k\right) - \bm S^k \right\| \leq \| \alpha_k \X^k\|$, we know 
 $\| \Delta \bm W^k \| \leq \left\|\frac{\alpha_k}{2} \X^k\right\| \leq \frac{\alpha_{\max}}{2} \eta_{\X} =: \eta_{\Delta \bm W}$ and 
$\| \Delta \bm S^k \| \leq \left\|{\alpha_k} \X^k\right\| \leq \alpha_{\max} \eta_{\X} =: \eta_{\Delta \bm S}$.
\end{proof}

If $\|\bm D^k\| = 0$, 
we can say that $\bm U^k$ is 
an optimal solution of \eqref{dual} due to Theorem~\ref{THEO:optimality}, and Algorithm~\ref{alg} also terminates at Step 1. 
Therefore, we can assume that $\|\bm D^k\| > 0$ 
without loss of generality during the iterations of Algorithm~\ref{alg}.

The termination status of Algorithm~\ref{alg} can be 
divided into two cases, (i) the step length $\nu_k$ converges to zero before reaching 
divided into two cases, (i) the step length $\nu_k$ converges to zero before reaching 
an optimal solution 
(ii) Algorithm~\ref{alg} will stop at the optimal value, or generate a sequence that converges to the optimal value.

The case (i) will be denied by 
Lemma~\ref{LEMM:steplength-bound} below.
The proof of this lemma is similar to Lemma 7 in \cite{nakagaki2020dual}, but we need different constants like $\eta_{\Delta \bm S}$.
Therefore, the only possibility is the case (ii), and the 
convergence to the optimal value will be guaranteed 
in Theorem~\ref{theo:dspg}.
The proof of this lemma is similar to Lemma 7 in \cite{nakagaki2020dual}, but we need different constants like $\eta_{\Delta \bm S}$.
Therefore, the only possibility is the case (ii), and the 
convergence to the optimal value will be guaranteed 
in Theorem~\ref{theo:dspg}.

\begin{lemma}[Lower bound of step length]\label{LEMM:steplength-bound}
The step length $\nu_k$ of Algorithm~\ref{alg} 
has a positive lower bound.
\end{lemma}

\begin{proof}
We consider that
\begin{align}
\|\B(\bm D^k)\| & = \left\|-\A^\top (\Delta \bm y) + \frac{\Delta \bm W}{2} + \Delta \bm S\right\| \\
& \leq \|\A^\top\|\|\Delta \bm y^k\| + \frac{\| \Delta \bm W^k \|}{2} + \| \Delta \bm S^k \| \\
& \leq \max \left\{ \|\A^\top\|, 1 \right\} \left(\|\Delta \bm y^k\| + \| \Delta \bm W^k \| + \|\Delta \bm S^k \|\right) \label{eq:boundBD-1} \\
& \leq \sqrt{3}\max \left\{ \|\A^\top\| , 1 \right\} \|\bm D^k\|.\label{eq:boundBD}
\end{align} 
From Lemma~\ref{LEMM:bound}, we apply 
$\mu(\X^k)^{-1} \succeq \beta_{\min} \mu \I, \|\Delta \y^k\| \leq \eta_{\Delta \y}, \|\Delta \bm W^k\| \leq \eta_{\Delta \bm W}, \|\Delta \bm S^k \| \leq \eta_{\Delta \bm S}$, and we can obtain $\|\B(\bm D^k)\|$ can be bounded;
$\|\B(\bm D^k)\| \le \eta_{\B(\D)} := \max \left\{ \|\A^\top\| , 1 \right\}(\eta_{\Delta \y} + \eta_{\Delta \bm W} + \eta_{\Delta \bm S})$ from 
\eqref{eq:boundBD-1}
and $\|\B(\bm D^k)\| \le 
\sqrt{3}\max \left\{ \|\A^\top\|, 1 \right\} \|\bm D^k\|$
from \eqref{eq:boundBD}.

Our first goal is to show that $\nu_k$ from Step 2 of Algorithm~\ref{alg} has a lower bound. If $\theta \geq 0$, $\nu_k = 1$ is the fixed value, so we consider only the case that $\theta < 0$. Since $\theta$ is the minimum eigenvalue of $\bm L_k^{-1}\B(\bm D^k) \left(\bm  L_k^\top\right)^{-1}$, $\theta$ is also  the maximum value such that $\B(\bm D^k) \succeq \theta (\C + \B(\bm U^k))$. This implies that $\nu = -\frac{1}{\theta}$ is the maximum value that satisfies $\C + \B(\bm U^k) + \nu\B(\bm D^k) \succeq \O$. 

Next, we consider the bound of positive $\nu$ that satisfies $\C + \B(\bm U^k) + \nu\B(\bm D^k) \succeq \O$. From the upper bound $\C + \B(\bm U^k) = \mu (\X^k)^{-1} \succeq \frac{\mu}{\beta_{\max}} \I$  and $\|\B(\bm D^k)\| \leq \eta_{\bm \B(\bm D)}$, we have
\begin{align}
\C + \B(\bm U^k) + \nu \B(\bm D^k) \succeq \frac{\mu}{\beta_{\max}} \I - \nu \|\B(\bm D^k)\| \I \succeq \left(\frac{\mu}{\beta_{\max}} - \nu \eta_{\B (\bm D)}\right) \I.
\end{align} 
Therefore, for any 
$\nu \in \left[0 , \frac{\mu}{\beta_{\max} \eta_{\B(\bm D)}}\right]$,
$\C + \B(\bm U^k) + \nu \B(\bm D^k) \succeq 0$ is satisfied. 
This implies that  
$- \frac{1}{\theta} \geq \frac{\mu}{\beta_{\max}} \eta_{\B(\bm D)}$ 
when $\theta < 0$. Now we can obtain a lower bound of $\nu_k$ by the following inequality,
\begin{align}
 \nu_k &= \min\left\{1, -\frac{\tau}{\theta}\right\} \geq \min\left\{1, \frac{\tau \mu}{\beta_{\max}} \eta_{\B(\bm D)}\right\} =: v_{\min}.  
\end{align}

Next, we will show that for any $\nu \in (0,\nu_k)$, $\C + \B(\bm U^k) + \nu \B(\bm D^k)$ is bounded below by $\C + \B(\bm U^k)$. If $\theta \geq 0$, $\B(\D^k) \succeq 0$ and it is obvious that $\C + \B(\bm U^k) + \nu \B(\bm D^k) \succeq \C + \B(\bm U^k) \succeq (1 - \tau) (\C + \B(\bm U^k))$. If $\theta < 0$, from the definition $\nu_k = \min\{1, -\tau/\theta\}$, $\C + \B(\bm U^k) + \nu \B(\bm D^k) \succeq \C + \B(\bm U^k) + \nu \theta (\C + \B(\bm U^k)) \succeq (1 - \tau)(\C + \B(\bm U^k))$. We use the lower bound $\C + \B(\bm U^k) = \mu (\X(\U^k))^{-1} \succeq \frac{\mu}{\beta_{\max}} I$ to imply $\C + \B(\bm U^k) + \nu \B(\bm D^k) \succeq (1 - \tau)\frac{\mu}{\beta_{\max}} \bm I$.

From Lemma 6 (iii) in \cite{nakagaki2020dual}, if $\overline{\bm X} \succeq \beta \bm I, \overline{\bm Y} \succeq \beta \bm I$ then $|| \overline{\bm Y} - \overline{\bm X} || \geq \beta^2 || \overline{\bm Y}^{-1} - \overline{\bm X}^{-1}||$. Therefore, from $X(\bm U^k + \nu \bm D^k) = \mu(\C + \B(\bm U^k) + \nu \B(\bm D^k))^{-1}$ , we obtain
\begin{align}
    & \ \left\|\left(\bm X(\bm U^k + \nu \bm D^k) - \bm X(\bm U^k) \right)\right\|  = \mu \|(\C + \B(\bm U^k) + \nu \B(\bm D^k))^{-1} - (\C + \B(\bm U^k))^{-1}\|\\
		\leq & \ \frac{\mu \|\nu\B(\bm D^k)\|}{((1-\tau)\frac{\mu}{\beta_{\max}})^2} 
     = \frac{\nu \|\B(\bm D^k)\|}{\mu(\frac{1-\tau}{\beta_{\max}})^2}.
\end{align}

We are now in a position to show that $\nabla g$ is a Lipschitz continuity for the direction $\bm D^k$. 

\begin{align}
    & \ \|\nabla g(\bm U^k + \nu \bm D^k) - \nabla g(\bm U^k)\|\\
    = & \ \left\|\left(-\A(\bm X(\bm U^k + \nu \bm D^k) - \bm X(\bm U^k)), \frac{1}{2} \bm X(\bm U^k + \nu \bm D^k) - \frac{1}{2}\bm X(\bm U^k), \bm X(\bm U^k + \nu \bm D^k) - \bm X(\bm U^k)\right)\right\|\\
    \leq & \ \sqrt{\|\A\|^2 + \frac{5}{4}}\|X(\bm U^k + \nu \bm D^k) - \bm X(\bm U^k)\| \\
    \leq & \ \frac{\nu \sqrt{\|\A\|^2 + \frac{5}{4}}}{\mu(\frac{1-\tau}{\beta_{\max}})^2} \|\B(\bm D^k)\|\\
    \leq & \ \frac{\sqrt{\|\A\|^2 + \frac{5}{4}}\sqrt{3}\max \left\{ \|\A\|, 1\right\}}{\mu(\frac{1-\tau}{\beta_{\max}})^2} \nu  \|\bm D^k\|,
\end{align}
where the last inequality came from \eqref{eq:boundBD}. We can conclude that $\|\nabla g(\bm U^k + \nu \bm D^k) - \nabla g(\bm U^k)\|$ is bounded by $\|\nu \bm D^k\|$, which shows the Lipschitz continuity with the Lipschitz constant $L := \frac{\sqrt{\|\A\|^2 + \frac{5}{4}}\sqrt{3}\max \left\{ \|\A\|, 1 \right\}}{\mu(\frac{1-\tau}{\beta_{\max}})^2}$.

In the last part, We consider the termination condition at Step 2 of Algorithm~\ref{alg}. 
When it terminates at the first iteration ($\sigma_k = 1$), 
it holds that $\sigma_k \nu_k = \nu_k \geq \nu_{\min}$.  
If it terminates at $\sigma_k = \beta^j$, then the condition is not satisfied at $\sigma_k = \beta^{j-1}$, thus,
\begin{align}
\ g(\bm U^k + \beta^{j-1}\nu_k \bm D^k) < & \ \min \limits_{[k-M]_{+}\leq j \leq k} g(\bm U^j) + \gamma \beta^{j-1}\nu_k \langle \nabla g(\bm U^k) , \bm D^k \rangle \\
 \leq & \ g(\bm U^k) + \gamma \beta^{j-1}\nu_k \langle \nabla g(\bm U^k) , \bm D^k \rangle.
\label{eq:lowerg}
\end{align}
From Taylor's expansion, we obtain
\begin{align}
& \ g(\bm U^k + \beta^{j-1}\nu_k \bm D^k)^k - g(\bm U^k) \\
=& \ \langle \beta^{j-1}\nu_k \nabla g(\bm U^k) , \bm D^k \rangle
+ \int_0^{\beta^{j-1}\nu_k} \langle \nabla g(\bm U^k + \lambda \bm D^k )- g(\bm U^k) , \bm D^k \rangle d \lambda. \label{eq:middleg}
\end{align}
Since $\nabla g$ is a Lipschitz continuity for the direction $\bm D^k$,
\begin{align}
     \langle \nabla g(\bm U^k + \lambda \bm D^k )- g(\bm U^k) , \bm D^k \rangle \geq - \| \nabla g(\bm U^k + \lambda \bm D^k )- g(\bm U^k)\|\|\bm D^k\| \geq -L\lambda\|\bm D^k\|^2.
\end{align}
Thus,
\begin{align}
 \int_0^{\beta^{j-1}\nu_k} \langle \nabla g(\bm U^k + \lambda \bm D^k )- g(\bm U^k) , \bm D^k \rangle d \lambda 
\ \geq   - \int_0^{\beta^{j-1}\nu_k} L\lambda\|\bm D^k\|^2 d \lambda 
\ =   \frac{-L(\beta^{j-1}\nu_k)^2}{2} \|\bm D^k\|^2.\label{eq:upperg}
\end{align}

Combining \eqref{eq:lowerg} , \eqref{eq:middleg}, and \eqref{eq:upperg} implies
\[
\beta^{j-1}\nu_k \geq \frac{2(1 - \gamma)}{L}\frac{\langle \nabla g(\bm U^k) , \bm D^k \rangle}{\|\bm D^k\|^2}.
\]
In Lemma \ref{LEMM:bounded}, we obtained $\langle \nabla g(\bm U^k) , \bm D^k \rangle \geq \frac{\|\bm D^k\|^2}{\alpha_k} $, therefore,
\[
\frac{\langle \nabla g(\bm U^k) , \bm D^k \rangle}{\|\bm D^k\|^2} \geq \frac{1}{\alpha_k} \geq \frac{1}{\alpha_{\max}}.\
\]
Finally, since the step length is $\sigma_k \nu_k = \beta^{j}\nu_k$, we can conclude that $\beta^{j}\nu_k \geq \frac{2 \beta (1 - \gamma)}{L \alpha_{\max}}$, which means that the step length $\sigma_k \nu_k$ has a lower bound. This completes the proof.
\end{proof}

Let $(\sigma\nu)_{\min} := \min\{\nu_{\min}, \frac{2 \beta (1 - \gamma)}{L \alpha_{\max}}\}$ denote the positive lower bound of the step length $(\sigma_k \nu_k)$. In the next lemma, we will use this result to show that a subsequence of the search direction converges to $0$. 

\begin{lemma}\label{lemm:liminf-delta}

Algorithm~\ref{alg} with $\epsilon = 0$ stops after reaching the optimal value $g^*$, or
\[
\liminf_{k\to \infty}\|\bm R^k\| = 0.
\]
\end{lemma}

\begin{proof}
Firstly, we will show that $\|\bm R^k\|$ is bounded by $\|\bm D^k\|$. From the property of the projection in ~\cite{HAGER08}, we know that $\|P_{\M}\left(\bm U^k + \alpha \nabla g(\bm U^k)\right)  - \bm U^k\|$ is a non-decreasing function and $\|P_{\M}\left(\bm U^k + \alpha \nabla g(\bm U^k)\right)  - \bm U^k\| / \alpha$ is a non-increasing function for $\alpha > 0$. This implies $\|\bm R^k\| \leq \|\bm D^k\|$ for $\alpha \leq 1$, and $\|\bm R^k\| \leq \alpha \|\bm D^k\|$ for $\alpha > 1$. We can conclude that $\|\bm R^k\| \leq \max\{1,\alpha_{\max}\}\|\bm D^k\|$.
Therefore, it is enough to show $\liminf_{k\to \infty}\|\D^k\| = 0$ instead. 
Therefore, it is enough to show $\liminf_{k\to \infty}\|\D^k\| = 0$ instead. 

Suppose that there exists $\delta >0$ and $k_0 > M$ such that $\|\bm D^k\| > \delta$ for any $k > k_0$ and we derive a contradiction. From the proof of Lemma~\ref{LEMM:bounded}, we obtain $\left\langle\nabla g(\bm U^k),\,\bm D^k\right\rangle\ge \frac{\|\bm D^k\|^2}{\alpha_k} \ge \frac{\|\bm D^k\|^2}{\alpha_{\max}}$. Combining with the lower bound of step length in Lemma~\ref{LEMM:steplength-bound}, we can
show that  $\gamma\sigma_k \nu_k\langle\nabla g(\bm U^k),\,\bm D^k\rangle$ has a lower bound;
$\gamma\sigma_k \nu_k\langle\nabla g(\bm U^k),\,\bm D^k\rangle \geq \gamma \sigma_k \nu_k \frac{\|\bm D^k\|^2}{\alpha_{\max}} \ge \gamma (\sigma \nu)_{\min} \frac{\delta^2}{\alpha_{\max}} =: \hat{\delta}$. 
Let $g_k^{\min} := \min_{k - M \leq j \leq k} g(\bm U^j)$. Following the condition in Step 2 of Algorithm~\ref{alg}, we obtain
\begin{equation}
    g(\bm U^{k+1}) \geq g_k^{\min} + \hat{\delta}.
\end{equation}
This follows by
\begin{align}
g(\bm U^{k+2})  \geq \min_{k - M + 1 \leq j \leq k + 1} g(\bm U^j) + \hat{\delta} 
 \geq \min\{g(\bm U^{k+1}),g_k^{\min}\} + \hat{\delta} \geq g_k^{\min} + \hat{\delta}.
\end{align}
Using the induction, we can derive that for $j = 1,2,\dots, M$, 
\begin{equation}
    g(\bm U^{k+j}) \geq g_k^{\min} + \hat{\delta}.
\end{equation}
This leads to
\begin{equation}
    g_{k+M}^{\min} =  \min_{k + 1 \leq j \leq k + M} g(\bm U^j)\geq g_k^{\min} + \hat{\delta}.
\end{equation}
The above statement is true for all $k > k_0$, so repeating the inequality $P$ times leads to 
\begin{equation}
    g_{k+ PM}^{\min} \geq g_k^{\min} + P\hat{\delta}.
\end{equation}
On the other hand, 
the level set $\mathcal{L}$ is bounded and closed, therefore, the dual problem has a finite optimal value. Consequently, the sequence $\{g(\bm U^k)\}$ should be bounded by the optimal value $g^*$ from the above, and we have a contradiction. This completes the proof.
\end{proof}

To show the convergence of the objective value ($\liminf_{k \to \infty} |g(\bm U^k) - g^*| = 0$) in
Lemma~\ref{lemm:liminf-obj} below, we need more upper bounds. 
\begin{lemma}\label{lemm:upperbound}
$|\sum_{i<j}\rho |X_{ij}| - \frac{\bm W^k}{2} \bullet \bm X^k|$ is bounded by $\|\Delta \bm W^k_{(1)}\|$, and $|\sum_{i<j,s<t}\lambda |X^k_{ij} - X^k_{st}| - \bm S^k \bullet \bm X^k|$ is bounded by $\|\Delta \bm S^k_{(1)}\|$. 
\end{lemma}
\begin{proof}
We can prove the first statement by following the proof of Lemma~9 in \cite{nakagaki2020dual}, thus we focus on the second statement. 
Due to the definition of the linear map $\Q$, we have $\sum_{i<j}\sum_{s<t} |X_{ij} - X_{st}| = ||\Q(\bm X)||_1$. Let $\E^k \in {\SMAT}^{\frac{n(n-1)}{2}}$ be the sign matrix of $\Q (\bm X^k)$ defined by $\E^k_{ij} = \text{sign}([\Q(\bm X^k)]_{ij})$
and $\hat{\bm S}^k = P_{\SC}(\bm S^k + \bm X^k)$. Then, 
\begin{align}
 & \ |\lambda\|\Q(\X^k)\|_{1} - \bm S^k \bullet \X^k| 
=  \ |\lambda \E^k \bullet \Q(\X^k) - \bm S^k \bullet \X^k| 
=  \ |\lambda\Q^\top (\E^k) \bullet \X^k - \bm S^k \bullet \X^k | \\ 
\leq & \ |\lambda\Q^\top (\E^k) \bullet \X^k - \hat{\bm S}^k \bullet \X^k| + |(\hat{\bm S}^k - \bm S^k) \bullet \X^k| 
=  \ |(\lambda \Q^\top (\E^k) - \hat{\bm S}^k) \bullet \X^k| + |\Delta \bm S^k_{(1)} \bullet \X^k|.\label{bound2}
\end{align}
The last equality was derived from Remark~\ref{REMA:X-U}
and $\Delta \bm S^k_{(1)} = P_{\SC}(\bm S^k + \bm X^k) - \bm S^k$.

We can show the bound of the second term in \eqref{bound2} by $|\Delta \bm S^k_{(1)} \bullet \X^k| \leq ||\Delta \bm S^k_{(1)}||||\X^k|| \leq \eta_{\X}||\Delta \bm S^k_{(1)}||$. Therefore, we will focus on the first term. 
 Since $\lambda \Q^\top (\E^k) = \Q^\top (\lambda \E^k) \in \SC$ and $\SC$ is a convex set, 
 a property of the projection (Proposition 2.1~(1) in \cite{HAGER08}) leads to 
\begin{align}
\left((\bm S^k + \X^k) - \hat{\bm S}^k\right) \bullet \left(\lambda\Q^\top (\E^k) - \hat{\bm S}^k\right)
= \left(\X^k - \Delta \bm S^k_{(1)}\right) \bullet \left(\lambda\Q^\top (\E^k) - \hat{\bm S}^k\right)
\leq 0.
\end{align}
This indicates
\begin{align}\label{subbound2.5}
\X^k \bullet (\lambda \Q^\top (\E^k) - \hat{\bm S}^k) \leq \Delta \bm S^k_{(1)} \bullet (\Q^\top(\E^k) - \hat{\bm S}^k).
\end{align}
From the property of projection onto $\SC$, there exists $\hat{\V} \in \SMAT^{\frac{n(n-1)}{2}}$ such that $\|\hat{\V}^k\|_{\infty} \leq \lambda$ and $\hat{\bm S}^k = \Q^\top (\hat{\V}^k)$.
Now we examine the value of $\X^k \bullet (\lambda \Q^\top (\E^k) - \hat{\bm S}^k) = \Q(\X^k) \bullet (\lambda \E^k - \hat{\V}^k) = \sum_{i = 1}^{\frac{n(n-1)}{2}}\sum_{j = 1}^{\frac{n(n-1)}{2}} [\Q(\X^k)]_{ij} (\lambda E^k_{ij} - \hat{V}^k_{ij})$.
    We will divide the value of $[\Q(\X^k)]_{i}$ into three cases.
\begin{enumerate}
\item[Case 1:] $[\Q(\X^k)]_{ij} = 0$ \\
In this case, we obtain $[\Q(\X^k)]_{ij} (E^k_{ij} - \hat{V}_{ij}) = 0$.
\item[Case 2:] $[\Q(\X^k)]_{ij} > 0$ \\
By the definition of $\E^k$, $E^k_{ij} = 1$, which means $\lambda E^k_{ij} - \hat{V}_{ij} \geq 0$ because $|{\hat{V}}^k_{ij}| \leq \lambda$. This implies $[\Q(\X^k)]_{ij} (\lambda \E^k_{ij} + \hat{\V}_{ij}) \geq 0$.
\item[Case 3:] $[\Q(\X^k)]_{ij} < 0$ \\
By the definition of $\E^k$, $E^k_{ij}= -1$, which means $\lambda E^k_{ij} - \hat{V}_{ij} \leq 0$ because $|{\hat{V}}^k_{ij}| \leq \lambda$. This implies $[\Q(\X^k)]_{ij} (\lambda \E^k_{ij} - \hat{\V}_{ij}) \geq 0$.
\end{enumerate}
Therefore, it holds that 
\begin{align}
    \X^k \bullet (\lambda \Q^\top (\E^k) - \hat{\bm S}^k) = \Q(\X^k) \bullet (\lambda \E^k - \hat{V}^k) \geq 0.
\end{align}

Applying this result to \eqref{subbound2.5}, we obtain
\begin{align}\label{bound3}
|\X^k \bullet (\lambda \Q^\top (\E^k) - \hat{\bm S}^k)| \leq  \ |\Delta \bm S^k_{(1)} \bullet (\lambda \Q^{T}(\E^k) - \hat{\bm S}^k)|,
\end{align}
which means $|\X^k \bullet (\Q^\top (\E^k) - \hat{\bm S}^k)|$ is bounded by $\|\Delta \bm S^k_{(1)}\|$, 
since $\hat{\bm S}^k \in \SC$ and $\SC$ is bounded.
Combining \eqref{bound2} and \eqref{bound3}, we can conclude this lemma.
\end{proof}

\begin{lemma}\label{lemm:liminf-obj}
Algorithm~\ref{alg} with $\epsilon = 0$ stops after reaching the optimal value $g^*$, or generate a sequence $\{\bm U^k\} \subset \FC$ such that
\[
\liminf_{k\to \infty}|g(\bm U^k) - g^*| = 0.
\]
\end{lemma}

\begin{proof}
We split the left-hand side of the objective equation into three parts as the following inequality:
\begin{equation}\label{divideparts}
|\g(\bm U^k) - g^*|  \leq |\g(\bm U^k) - f(\X^k)| + |f(\X^k) - f(\X^*)| + |f(\X^*) - g^*|,
\end{equation}
where $\X^*$ is the optimal solution of the primal problem
\eqref{primal}.
We will show that each term is bounded by $\|\bm R^k\|$. Remind that $\bm X^k : = \bm X(\bm U^k) = \mu \left(\bm C + \B(\bm U^k)\right)^{-1}$. Since  the third term of \eqref{divideparts} is equal to zero by the duality theorem, 
we focus on the bounds of the first and second terms. The first term of \eqref{divideparts} can be bounded by 
\begin{align}
    & \ |g(\bm U^k) - f(\X^k)| \\
    = & \Bigg| \left(\bm b^\top\bm y^k  + \mu\log\det\left(\bm C - \A^\top(\bm y^k) + \frac{\bm W^k}{2} + \bm S^k\right) + n \mu - n \mu \log \mu \right) \\
    & -  \left( \bm C \bullet \bm X^k - \mu \log \det \bm X^k + \rho \sum_{i < j}|X^k_{ij}| + \lambda \sum_{i < j}\sum_{s < t}|X^k_{ij} - X^k_{st}| \right)\Bigg| \\
    = & \ \Bigg| \left(\frac{\bm W^k}{2} \bullet \bm X^k - \sum_{i<j}\rho |X_{ij}^k|\right) + \left( \bm S^k \bullet \bm X^k - \sum_{i<j,s<t}\lambda |X^k_{ij} - X^k_{st}|\right) + (\bm b - \AC(\X^k))^T \y \Bigg| \\
    \leq & \ \left|\frac{\bm W^k}{2} \bullet \bm X^k - \sum_{i<j}\rho |X_{ij}^k|\right| + \left| \bm S^k \bullet \bm X^k - \sum_{i<j,s<t}\lambda |X^k_{ij} - X^k_{st}|\right| + \|\bm y^k \|\|\AC\|\|\X^* - \X^k\|.\label{split1}
\end{align}
From Lemma~\ref{lemm:upperbound}, the sum of the first and second terms of \eqref{split1} is bounded by 
$\|\Delta \bm W_{(1)}^k\|$ and $\|\Delta \bm S_{(1)}^k\|$, respectively. Thus, they are also bounded by
$\|\bm R^k\| = \sqrt{||\bm y^k||^2 + \|\Delta \bm W_{(1)}^k\|^2 + \|\Delta \bm S_{(1)}^k\|^2}$.
From Lemma~\ref{LEMM:bounded}, $\|\bm y^k\|$ is bounded, and $\|\AC\|$ is also a constant. Therefore, 
the third term in \eqref{split1} is bounded by $\|\X^k - \X^*\|$. Combining the results, we can show that
\begin{equation}\label{boundterm1}
    |g(\bm U^k) - f(\X^k)| \leq c_1 ||\bm R^k|| + c_2 \|\X^k - \X^*\|
\end{equation}
for some positive constants $c_1$ and $c_2$.

Next, we show that $|f(\X^k) - f(\X^*)|$ is bounded by $\|\X^k - \X^*\|$ by splitting $|f(\X^k) - f(\X^*)|$ into the following four terms:
\begin{equation}
\begin{split}
    |f(\X^k) - f(\X^*)| \leq & \ \left|\C \bullet (\X^k - \X^{*})\right| + \mu \left|(\log \det \X^k - \log \det \X^{*})\right| + \left|\rho \sum_{i<j}(|X^k_{ij}| - |X^*_{ij}|)\right|\\ 
    & \ + \lambda \left|\sum_{i<j} \sum_{s<t} (|X^k_{ij} - X^k_{st}| - |X^*_{ij} - X^*_{st}|)\right|,
\end{split}
\end{equation}
and show that each term is bounded by $\|\X^k - \X^*\|$. 

Since $\bm C$ is an input matrix, the boundedness of the first term is obvious. For the second term, using the convexity of function $f_2(\X) = - \log \det \X$, we obtain
\begin{align}
    f_2(\X^k) \leq & \ f_2(\X^*) + \nabla f_2(\X^*) \bullet (\X^k - \X^*) \\
    f_2(\X^*) \leq & \ f_2(\X^k) + \nabla f_2(\X^k) \bullet (\X^* - \X^k),
\end{align}
which imply
\begin{align}
    & \ |\log \det \X^k - \log \det \X^{*}|  = |f_2(\X^k) - f_2(\X^*)| \\
     \leq & \ \max \left\{ \|\nabla f_2(\X^*)\| , \|\nabla f_2(\X^k)\|\right\} \|(\X^k - \X^*)\| 
     = \max\left\{\|(\X^*)^{-1}\| , \|(\X^k)^{-1}\|\right\}\|(\X^k - \X^*)\|.
\end{align}
and it is bounded by $\|\X^k - \X^*\|$ due to  $\|(\X^k)^{-1}\| \le \eta_{{\bm X}^{-1}}$ 
in Lemma~\ref{LEMM:bound}. 
For the third term, we have 
\begin{align}\left|\sum_{i<j}(|X^k_{ij}| - |X^*_{ij}|)\right| \leq \sum_{i<j}||X^k_{ij}| - |X^*_{ij}|| \leq \sum_{i<j}|X^k_{ij} - X^*_{ij}| \leq \frac{\sqrt{n(n-1)}}{2} \|\X^k - \X^*\|,\end{align}
and, for the fourth term, we have
\begin{align}
& \ \left|\sum_{i<j} \sum_{s<t} (|X^k_{ij} - X^k_{st}| - |X^*_{ij} - X^*_{st}|)\right|
\leq \ \sum_{i<j} \sum_{s<t} \left||X^k_{ij} - X^k_{st}| - |X^*_{ij} - X^*_{st}|\right|\\
\leq & \ \sum_{i<j} \sum_{s<t} \left|X^k_{ij} - X^k_{st} - X^*_{ij} + X^*_{st}\right| 
\leq  \sum_{i<j} \sum_{s<t} \left(|X^k_{ij} - X^*_{ij}| + |X^k_{st} - X^*_{st}|\right) \\
= & \ n(n-1) \sum_{i<j} |X^k_{ij} - X^*_{ij}| 
\leq \frac{\left(n(n-1)\right)^{\frac{3}{2}}}{2} \|\X^k - \X^*\|.
\end{align}
Therefore, we obtain that $|f(\X^k) - f(\X^*)|$ is bounded by $\|\X^k - \X^*\|$. In other words, there is a constant $c_3$ such that
\begin{equation}\label{boundterm2}
    |f(\X^k) - f(\X^*)| \leq c_3 \|\X^k - \X^*\|.
\end{equation}
Next, we employ similar steps to Lemma 10 in \cite{nakagaki2020dual} 
to derive 
\begin{equation}\label{bounddifX}
   \|\X^k - \X^*\| \leq h_1(\| \bm R^k \|), 
\end{equation}
such that $h_1(x) \ge 0$ for any $x \ge 0$ and $\lim_{x \to 0^+} h_1(x) = 0$.

Lastly, applying \eqref{boundterm1} and \eqref{boundterm2} into \eqref{divideparts}, we can conclude that
\begin{align}
\liminf_{k\to \infty}|g(\bm U^k) - g^*|  & \ \leq \liminf_{k\to \infty}\left(c_1 || \bm R^k || + (c_2 + c_3) \|\X^k - \X^*\|\right) \\
& \ \leq \liminf_{k\to \infty} \ \left(c_1 || \bm R^k || + (c_2+c_3) \cdot h_1(|| \bm R^k ||)\right). 
\end{align}
From Lemma~\ref{lemm:liminf-delta} and the definition of function $h_1$, both terms are converge to zero, which means that $\liminf_{k\to \infty}|g(\bm U^k) - g^*| = 0$. This completes the proof.
\end{proof}

Based on the above preparation, we are now in the position to show the convergence of Algorithm~\ref{alg}.

\begin{theorem}\label{theo:dspg}
Algorithm~\ref{alg} with $\epsilon = 0$ stops after reaching the optimal value $g^*$, or generates a sequence $\bm U^k$ such that
\[
\lim_{k\to \infty}|g(\bm U^k) - g^*| = 0.
\]
\end{theorem}

\begin{proof}
We will derive a contradiction to show this theorem. Suppose that there is $\epsilon > 0$ such that we have infinite sequence $\{k_1, k_2, \dots\}$ such that $g(\bm U^{k_i}) < g^* - \epsilon$ for all positive integer $i$, and $g(\bm U^{j}) \ge g^* - \epsilon$ for all $j \notin \{k_1, k_2, \dots\}$.

We will show that the sequence should satisfy $k_{i+1} - k_i \leq M$. Suppose that $k_{i+1} - k_i > M$. Therefore, $g(\bm U^l) \geq g^* - \epsilon$ for all $k_i < l < k_{i+1}$. From the condition in Step 2 of Algorithm~\ref{alg} and 
$\langle\nabla g(\bm U^{k_{i+1}-1}),\,\bm D^{k_{i+1}-1}\rangle \ge 0$, we obtain
\begin{equation}
    g(\bm U^{k_{i+1}}) \geq \min_{k_{i+1} - M \leq l \leq k_{i+1}-1} g(\bm U^l) + \gamma\sigma_k \nu_k\langle\nabla g(\bm U^{k_{i+1}-1}),\,\bm D^{k_{i+1}-1}\rangle \geq g^* - \epsilon.
\end{equation}
and contradicts with the assumption $g(\bm U^{k_i}) < g^* - \epsilon$ for all $i$. Thus, the sequence $\{k_1, k_2, \dots\}$ should satisfy $k_{i+1} - k_i \le M$ and should be infinite. 

From Lemma~\ref{lemm:liminf-obj}, $|g(\bm U^k) - g^*| \leq h_2(|| \bm R^k||)$,
therefore, there is $\bar{\epsilon}$ such that $|| \bm R^{k_i}|| > \bar{\epsilon}$ for all $i$. Applying the same proof as Lemma \ref{lemm:liminf-delta}, $|| \bm R^{k_i}||$ has a lower bound, hence, $|| \bm D^{k_i}||$ has a lower bound $D_{\min} > 0$ due to \cite[Lemma 5]{nakagaki2020dual}. 
This leads to that $\gamma\sigma_{k_i} \nu_{k_i}\langle\nabla g(\bm U^{k_i}),\,\bm D^{k_i}\rangle \geq \gamma \sigma_{k_i} \nu_{k_i} \frac{\|\bm D^{k_i}\|^2}{\alpha_{\max}} > \gamma (\sigma\nu)_{\min} \frac{\|\bm D^{k_i}\|^2}{\alpha_{\max}}$, thus, $\langle\nabla g(\bm U^{k_i}),\,\bm D^{k_i} \rangle$ has a lower bound $\bar{\delta} := \gamma (\sigma \nu)_{\min} \frac{D_{\min}^2}{\alpha_{\max}} $. Using the definition $g_k^{\min} := \min_{k - M + 1 \leq j \leq k} g(\bm U^j)$, for any $k_i > M$, we have
\begin{equation}
    g(\bm U^{k_{i}}) \geq g^{\min}_{k_{i}-1} + \bar{\delta}.
\end{equation}

From the property of sequence $\{k_i\}$, we have $k_{i} - M \le k_{i-1} \le k_{i} - 1$. It follows that there exists a positive integer $l(i)$ such that $g(\bm U^{k_{l(i)}}) = g^{\min}_{k_{i}-1}$ and $k_i - k_{l(i)} \leq M$. It means that for any $k_i > M$, there exists $k_{l(i)} \ge k_i - M$ such that
\begin{equation}
    g(\bm U^{k_{i}}) \geq g(\bm U^{k_{l(i)}}) + \bar{\delta}.
\end{equation}
Repeating the inequality $P$ times, we can see that for any positive integer $P$ and $k_j > PM$, there exists $k_{l^{P}(j)} > k_j - PM$  such that
\begin{equation}
    g(\bm U^{k_{j}}) \geq g(\bm U^{k_{{l^P}(j)}}) + P\bar{\delta},
\end{equation}
where ${l^P}(j) := l(l(\cdots l(j) \cdots ))$, the $P$ repetition of $l$. From Lemma \ref{LEMM:bounded}, $g(\bm U^0)$ is a lower bound of $g(\bm U^{k_i})$. Therefore, when $g^*$ is the optimal value of the  dual problem~\eqref{dual} and we take $P > \frac{g^* - g(\bm U^0)}{\bar{\delta}}$ and $k_j > PM$, we get $g(\bm U^{k_j}) > g^*$, which contradicts the optimality of $g^*$. This completes the proof.

\end{proof}

\section{Computation Complexity}\label{subproblem}

In this section, we focus on the computation complexity of obtaining the projection
\begin{align}
P_{\mathcal{M}}(\bm U) = (\y, P_{\W}(\bm W), P_{\S}(\bm S)).
\end{align}
This is the main reason for the use of the variable $\bm S$. 
Recall that we can also apply the original DSPG to \eqref{dual_pre}
directly by using the variable $(\bm y,\bm W,\bm Z)$ 
and define the projection of $\bm Z$ by the box constraint like $\W$.
However, the number of constraints in $\| \bm Z \|_{\infty}$ amounts to approximately $\frac{n^4}{8}$, 
and this needs $\OC(n^4)$ operations. 
The main purpose of introducing $\bm S$ is to reduce 
the size of the variable matrix, 
and we will show that the cost of the projection onto $\SC$
can be reduced to $\OC(n^2 \log n)$. 

Let $\bar{n} = \frac{n(n-1)}{2}$ and $\s = \{s_1,s_2,\dots, s_{\bar{n}}\} = vect(\bm S)$.
Therefore, the subproblem for computing $P_{\S}(\bm S)$ can be reduced to the form: 
\begin{equation}\label{projectionS}
\begin{split}
 \min_{\bm u \in \Real^{\bar{n}}, \bm s \in \Real^{\frac{\bar{n} (\bar{n}-1)}{2}} } & \ \  \frac{1}{2} \sum_{i = 1}^{\bar{n}} (u_{i} - s_{i})^2\\
 \mbox{s.t.}  &  \ \ u_{i} = \sum_{j > i} z_{ij} - \sum_{j < i} z_{ji} \ \text{for} \ i =1,\ldots \bar{n} \\
 & \ \  |z_{ij}| \leq \lambda
\ \text{for} \ 1 \le i < j \le \bar{n}.
\end{split}
\end{equation}
Its dual problem is
\begin{equation}\label{projectionSDual}
\begin{split}
 \min_{\bm \pi \in\Real^{\bar{n}}} & \ \ p(\pib) := \frac{1}{2} \sum_{i = 1}^{\bar{n}} (\pi_{i} - s_{i})^2 + \lambda \sum_{i < j} |\pi_{i} - \pi_{j}|.
\end{split}
\end{equation}
Since these primal and dual problems are convex problems,
if $\u^*$ and $\pib^*$ are their optimal solutions, it holds that 
\begin{align}
    \u^* - \s = \pib^*.
\end{align}

Let $\s'$ be the vector whose components are the components of $\s$ in the non-decreasing order. We use $\T'$ to denote a permutation matrix $\T$ that corresponds to the order $\s'$, that is, $\s' = \T' \s$. We can rewrite the problem \eqref{projectionSDual} as
\begin{equation}\label{projectionSDual*}
\begin{split}
 \min_{\pi' \in\Real^{\bar{n}}} & \ \ p'(\pib') := \frac{1}{2} \sum_{i = 1}^{\bar{n}} (\pi_{i}' - s'_{i})^2 + \lambda \sum_{i < j} |\pi_{i}' - \pi_{j}'|.
\end{split}
\end{equation}

If  an optimal solution of \eqref{projectionSDual*} is $\pib'$, 
that of \eqref{projectionSDual} is given by $(\T')^{-1}\pib'$. 
Let $\pib''$ be the vector whose components are 
the components of $\pib'$ in the non-decreasing order.

\begin{lemma}\label{LEMM:nondecreasing}
If $\pib'$ is an optimal solution of \eqref{projectionSDual*}, $\pib''$ is also an optimal solution of \eqref{projectionSDual*}.
\end{lemma}

\begin{proof}
Suppose that there is an index $t$ such that $\pi'_t > \pi'_{t+1}$. Let $\tilde{\pib}$ be the vector obtained by swapping the components at the indexes $t$ and $t+1$ of $\pib'$. Therefore,
\begin{align*}
p'(\pib') - p'(\tilde{\pib}) &= \frac{1}{2} ((\pi'_t - s'_t)^2 + (\pi'_{t+1} - s'_{t+1})^2 - (\pi'_{t} - s'_{t+1})^2 - (\pi'_{t+1} - s'_t)^2) \\
&= - \pi'_t s'_t - \pi'_{t+1} s'_{t+1} + \pi'_t s'_{t+1} + \pi'_{t+1} s'_t \\
&= - (\pi'_t - \pi'_{t+1}) (s'_t - s'_{t+1}) \geq 0,
\end{align*}
where the equality holds if and only if $\s'_t = \s'_{t+1}$. This means that we can swap the value of $\pi'_t$ and $\pi'_{t+1}$ until we obtain $\pib''$, thus $\pib''$ is also an optimal solution.
\end{proof}

This lemma guarantees that there is an optimal solution $\pib'$ that satisfies the condition $\pi'_1 \leq \pi'_2 \leq \dots \le \pi'_{\bar{n}}$. We can add this condition into \eqref{projectionSDual*} as below:
\begin{equation}\label{projectionSDualOrder}
\begin{split}
\min_{\pi' \in\Real^{\bar{n}}} & \ \ p'(\pib') = \frac{1}{2} \sum_{i = 1}^{\bar{n}} (\pi'_{i} - s'_{i})^2 + \lambda \sum_{i < j} |\pi'_{i} - \pi'_{j}| \\
\mbox{s.t.}  &  \ \ \pi'_1 \leq \pi'_2 \leq \dots \pi'_{\bar{n}}.
\end{split}
\end{equation}
For any $\pib' \in \Real^{\bar{n}}$ in the non-decreasing order,
it holds that 
\[
\sum_{i<j} |\pi'_{i} - \pi'_{j}| = \sum_{i = 1}^{\bar{n}} (2i - \bar{n} - 1)\pi'_i.
\]
Therefore, it follows that
\begin{align*}
& \ \frac{1}{2} \sum_{i = 1}^{\bar{n}} (\pi'_{i} - s'_{i})^2 + \lambda \sum_{i < j} |\pi'_{i} - \pi'_{j}| 
=  \frac{1}{2} \sum_{i = 1}^{\bar{n}} (\pi'_{i}- s'_{i})^2 + \lambda \sum_{i = 1}^{\bar{n}} (2i - \bar{n} - 1)\pi'_{i} \\
= & \ \frac{1}{2} \sum_{i = 1}^{\bar{n}} \Big( (\pi'_{i} - s'_{i})^2 + 2\lambda (2i - \bar{n} - 1)\pi'_{i} \Big)
=  \ \frac{1}{2} \sum_{i = 1}^{\bar{n}} \Big(\pi'_{i} - (2s'_{i} - 2\lambda (2i - \bar{n} - 1))\pi'_{i} + {s'_{i}}^2\Big) \\
= & \ \frac{1}{2} \sum_{i = 1}^{\bar{n}} \Big(\big(\pi'_{i} - ( s'_{i}- \lambda (2i - \bar{n} - 1))\big)^2 \Big) + c,
\end{align*}
where $c = \frac{1}{2} \sum_{i=1}^{\bar{n}} \left(2 \lambda (2i - \bar{n} - 1) s'_{i} -\big( \lambda (2i - \bar{n} - 1)\big)^2\right)$ is a constant. Letting $s^{\dagger}_{i} = s'_{i}- \lambda (2{i} - n - 1)$, we rewrite \eqref{projectionSDualOrder} as
\begin{equation}\label{pavform}
\begin{split}
\min_{\pib'' \in\SMAT^{\bar{n}}} & \ \ p''(\pib^{\dagger}) := \frac{1}{2} \sum_{{i} = 1}^{\bar{n}} \big(\pi^{\dagger}_{i} - s^{\dagger}_{i}\big)^2 \\
\mbox{s.t.}  &  \ \ \pi^{\dagger}_1 \leq \pi^{\dagger}_2 \leq \dots \pi^{\dagger}_{\bar{n}}.
\end{split}
\end{equation}
The solution of \eqref{pavform} can be computed by using pool-adjacent-violators algorithm~\cite{PAVApaper,henzi2020accelerating} in $\OC(\bar{n}) = \OC(n^2)$ operations. 
After we obtain the optimal solution  of \eqref{pavform} as $\pib^{**}$, we can derive the optimal solution $\u^*$ of \eqref{projectionS} by $ \u^{*}= (\T')^{-1}\pib^{**} + \s$. Using this result, we have the following theorem.

\begin{theorem}\label{LEMM:Complexity}
We can compute the projection
\begin{align}
P_{\mathcal{M}}(\bm U) = (\y, [\bm W]_{\leq \rho}, [\bm S]_{\S}).
\end{align} in $\OC(n^2 \log n)$
\end{theorem}
\begin{proof}
For the complexity to obtain $P_{\W}(\bm W)$, we compute $\max\{ -\rho , \min \{ \rho, W_{ij}\}\}$ for each component of $\bm W$, Therefore, the complexity of $\OC(n^2)$ is enough for $P_{\W}(\bm W)$.
We divide the computation of $P_{\SC}(\bm S)$ into three parts.
The first part is extracting the upper-triangular part of $\bm S$ as the vector $\s = vect(\bm S)$ and sorting the components of $\s$ in the non-decreasing order, 
which has a complexity of $\OC(\bar{n} \log \bar{n}) = \OC(n^2 \log n)$. 
The second part is the pool-adjacent-violators algorithm, which has complexity $\OC(\bar{n}) = \OC(n^2)$. The third part is to convert the optimal solution of \eqref{projectionSDualOrder} back to the optimal solution of \eqref{projectionSDual}, which has complexity $\OC(\bar{n}) = \OC(n^2)$. Therefore, the total  complexity of computing the projection is
\begin{align}
\OC(n^2 \log n) + \OC(n^2) + \OC(n^2) = \OC(n^2 \log n).
\end{align}
\end{proof}

\section{Numerical Experiments}\label{numerical}

We conducted numerical experiments for Algorithm~\ref{alg}, 
DSPG~\cite{nakagaki2020dual} and Logdet-PPA~\cite{wang2010solving} on 
randomly generated synthetic data 
and a real animal dataset~\cite{kemp2008discovery}. 
All experiments are performed in Matlab R2022b on a 64-bit PC with Intel Core i7-7700K CPU (4.20 GHz, 4 cores) and 16 GB RAM.

For Algorithm~\ref{alg} and DSPG, we set the  parameters as 
$\gamma = 10^{-3}, \tau = 0.5 , \sigma = 0.5, \alpha_{\min} = 10^{-8}, \alpha_{\max} = 10^8$ and $M = 5$. 
We take an initial point 
$(\bm y^0, \bm W^0, \bm S^0) = (\0, \0, \0)$ for Algorithm~\ref{alg} 
and $(\bm y^0, \bm W^0, \bm Z^0) = (\0, \0, \0)$ for DSPG. For Logdet-PPA, we employed its default parameters.

The performance of each algorithm is evaluated with the number of iterations, the execution time, and a relative gap which is defined as 
\[
\text{Gap} = \frac{|P - D|}{\max\{ 1 , (|P| + |D|) / 2\}},
\]
where $P$ and $D$ are the output values of primal and dual objective functions, respectively. 
In addition, the convergence rate is evaluated
with a normalization error in each iteration by
\[
\text{E}^{k} = \frac{|D_k - D_{\textrm{opt}}|}{|D_0 - D_{\textrm{opt}}|},
\]
where $D_0, D_k, D_{\textrm{opt}}$ are objective values of $g$ at the initial point, the $k$th iteration, and the output point, respectively.

As the stopping criteria, 
we stopped DSPG and Algorithm~\ref{alg} when the iterate satisfies \\ 
$ 
    \|(\Delta \y_{(1)}^k, \Delta \W_{(1)}^k, \Delta \S_{(1)}^k)\|  \leq 10^{-16},
$ 
or reached 5000 iterations. Logdet-PPA was stopped when
$ 
    \max\{R_P , R_D\} \leq  10^{-6},
$ 
where $R_P$ and $R_D$ denote the primal and dual feasibility, respectively.
We chose these stopping criteria so that the three algorithms attained
a relative gap of about $10^{-7}$. 

\subsection{Randomly Generated synthetic Data}\label{sec:syntheticexp}

In this experiment, we generated input data following the procedure in 
\cite{nakagaki2020dual}. We first generated a sparse positive definite matrix 
$\BSigma^{-1} \in \SMAT^n$ with a density parameter $\sigma = 0.1$,
then 
we constructed the covariance matrix $\C \in \SMAT^n$ from $2n$ samples of the multivariate Gaussian distribution $\NC ( \0 , \BSigma ) $.
For a nonnegative integer $p$, let $\Omega_p := \{(i,j) \ | \  \Sigma^{-1}_{ij} = 0, \ 1 \le i < j \le n, \ |i - j| \leq p \}$.

We solved the following problem with unconstrained instances ($p = 0$) and constrained instances ($p = 2$ and $p=\left \lfloor{0.3*n}\right \rfloor$):
\begin{equation}\label{primalexp}
\begin{split}
 \min_{\bm X\in\bS^n} & \ \ f(\bm X) := \bm C \bullet \bm X - \mu \log \det \bm X + \rho \sum_{i < j}|X_{ij}| + \lambda \|\Q(\bm X)\|_1\\
\mbox{s.t.}  & \ \  X_{ij} = 0 \ \forall (i,j) \in \Omega_p,  \bm X \succ \bm 0.
\end{split}
\end{equation}
In the objective function, we employed the same weight parameters 
$\rho = \frac{5}{n}$ and $\lambda = \frac{\rho}{n(n-1)/2}$
as in \cite{nakagaki2020dual}.

We divide the experiments into two parts; the first is a comparison between three algorithms, while the second is the performance test of Algorithm~\ref{alg} with large instances.

Table~\ref{table:comparesyn-low} shows numerical results on problem \eqref{primalexp}.
The first column is the size $n$. The second, third, and fourth columns 
are the number of iterations, the computation time, and the relative gap for Algorithm~\ref{alg}. Similarly, the other six columns are for DSPG and Logdet-PPA.

\begin{table*}[tp] \centering \ra{1.3}
	\caption{Comparison of the performance for randomly generated synthetic data with small matrices}
	\label{table:comparesyn-low} 
	\addtolength{\tabcolsep}{-3pt}
\begin{tabular}{@{}rrrrrrrrrrrr@{}}\toprule 
	\multicolumn{12}{c}{$p=0$ (unconstrained case)}
	\\
	\hline
	&\multicolumn{3}{c}{Algorithm~\ref{alg}} & \phantom{abc}& \multicolumn{3}{c}{DSPG} & \phantom{abc} & \multicolumn{3}{c}{Logdet-PPA}\\ 
\cmidrule{2-4} \cmidrule{6-8} \cmidrule{10-12} $n$ & Iterations &  Time (s) & Gap &&  Iterations & Time (s) & Gap &&  Iterations & Time(s) & Gap \\
$10$ & 41 &  0.03 & 7.78e-9 && 350 & 0.11 & 1.50e-7 && 2 & 0.12 & 2.78e-7 \\ 
$20$ & 47 &  0.01 & 1.77e-8 && 1303 & 1.27 & 2.80e-7 && 14 & 1.77 & 2.11e-8\\
$25$ & 51 & 0.02 & 3.92e-9 && 1414 & 2.96 & 1.25e-7 && 84 & 19.64 & 1.86e-8\\
	\hline
	\multicolumn{12}{c}{$p=2$ (constrained case)} \\
	\hline		
	&\multicolumn{3}{c}{Algorithm~\ref{alg}} & \phantom{abc}& \multicolumn{3}{c}{DSPG} & \phantom{abc} & \multicolumn{3}{c}{Logdet-PPA}\\ 
\cmidrule{2-4} \cmidrule{6-8} \cmidrule{10-12} $n$ & Iterations & Time (s) & Gap &&  Iterations & Time (s) & Gap &&  Iterations & Time (s) & Gap \\
$10$ & 94 &  0.01 & 3.17e-8 && 338 & 0.06 & 1.29e-7 && 6 & 0.20 & 2.47e-6 \\ 
$20$ & 104 &  0.03 & 1.31e-8 && 1492 & 1.40 & 1.49e-7 && 19 & 2.30 & 1.66e-7\\
$25$ & 90 & 0.02 & 1.59e-9  && 1289 & 2.77 & 6.61e-8  && 84 & 19.40 & 2.47e-8\\
    \hline
	\multicolumn{12}{c}{$p = \left \lfloor{0.3*n}\right \rfloor\ $ (constrained case)} \\
	\hline		
	&\multicolumn{3}{c}{Algorithm~\ref{alg}} & \phantom{abc}& \multicolumn{3}{c}{DSPG} & \phantom{abc} & \multicolumn{3}{c}{Logdet-PPA}\\ 
\cmidrule{2-4} \cmidrule{6-8} \cmidrule{10-12} $n$ & Iterations & Time (s) & Gap &&  Iterations & Time (s) & Gap &&  Iterations & Time (s) & Gap \\
$10$ & 90 & 0.01 & 1.05e-7 && 454 & 0.07 & 2.74e-7 && 6 & 0.18 & 4.08e-7 \\ 
$20$ & 102 & 0.02 & 2.54e-8 && 1326 & 1.75 & 2.13e-7 && 51 & 5.25 & 7.92e-9\\
$25$ & 83 & 0.03 & 1.36e-8 && 1629 & 3.34 & 5.23e-8  && 102 & 20.97 & 8.95e-8\\
\bottomrule 
\end{tabular}
\end{table*}

We can observe from Table~\ref{table:comparesyn-low} that the proposed method 
(Algorithm~\ref{alg}) outperforms DSPG and Logdet-PPA
in the viewpoint of computation time in both unconstrained and constrained cases. 
For the unconstrained case $p=0$ and the size $n=25$, the proposed method 
solves the problem in 0.02 seconds, while DSPG and Logdet-PPA require
2.96 seconds and 19.64 seconds, respectively.
In addition, Algorithm~\ref{alg} is also efficient 
for constrained cases.
According to the structure of the last term in the objective function of \eqref{primalexp} (the term whose weight is $\lambda$), 
the number of elements in the summation grows rapidly when $n$ increases. 
Algorithm~\ref{alg} can deal with this problem better than DSPG and Logdet-PPA. 

Figure~\ref{fig:graph-compare-n25-2} displays the convergence rates of three methods in the case $n=25$ and $p = 2$. The horizontal axis of the graph is the computation time in seconds, and the vertical axis is the ${\text{E}}^k$ value. 
It is clear from the figure that the convergence speed of Algorithm~\ref{alg} is remarkably faster than those of DSPG and Logdet-PPA.
\begin{figure}[tp]
    \centering
    \includegraphics[clip, trim=1.5cm 9cm 1.5cm 9cm,scale = 0.6]{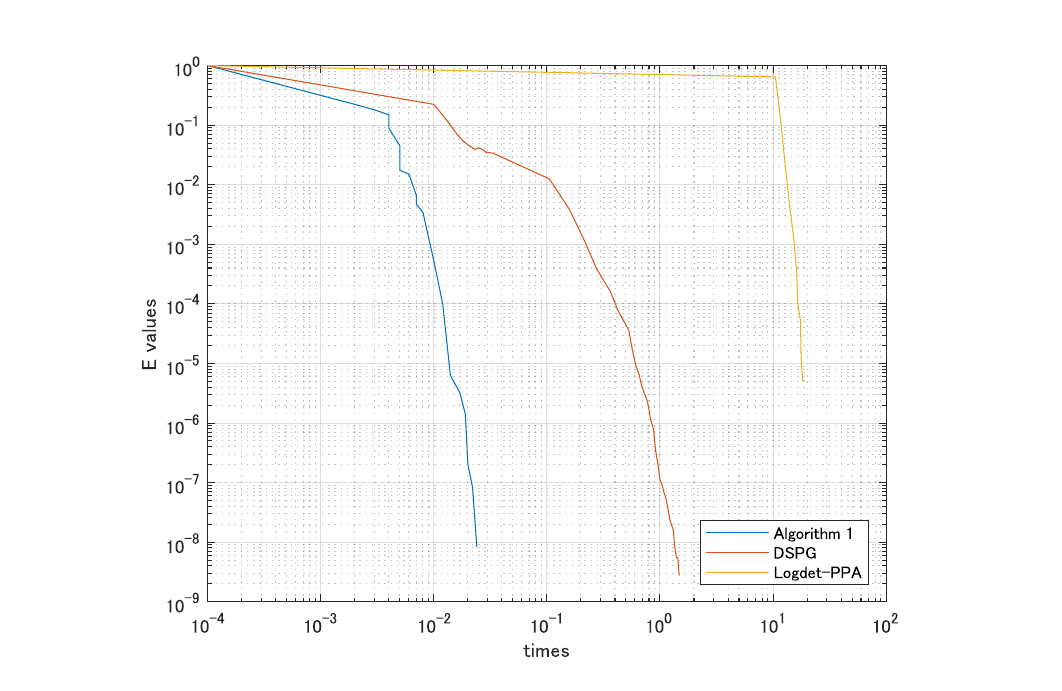}
    \caption{Comparison of the convergence rate for randomly generated synthetic data with $n = 25$ and $p = 2$}
    \label{fig:graph-compare-n25-2}
\end{figure}

Table~\ref{table:comparesyn-mid} shows numerical results on \eqref{primalexp} with medium matrices. We excluded Logdet-PPA from Table~\ref{table:comparesyn-mid}, since 
Logdet-PPA demanded more than 128 GB memory space for $n=30$.
\begin{table}[H] \centering \ra{1.3}
	\caption{Comparison of the performance for randomly generated synthetic data with medium matrices}
	\label{table:comparesyn-mid} 
\begin{tabular}{@{}rrrrrrrr@{}}\toprule 
	\multicolumn{8}{c}{$p=0$ (unconstrained case)}
	\\
	\hline
	&\multicolumn{3}{c}{Algorithm~\ref{alg}} & \phantom{abc} & \multicolumn{3}{c}{DSPG}\\ 
\cmidrule{2-4} \cmidrule{6-8} $n$ & Iterations & Time (s) & Gap &&  Iterations & Time (s) & Gap \\
$50$ & 59 & 0.06 & 2.85e-9 && 5000 & 284.83 & 1.46e-6\\ 
$75$ & 75 & 0.20 & 1.66e-8 && 5000 & 1430.53 & 1.36e-4\\
$100$ & 90 & 0.37 & 1.47e-8  && 5000 & 4352.69 & 1.48e-3\\
	\hline
	\multicolumn{8}{c}{$p=2$ (constrained case)} \\
	\hline		
		&\multicolumn{3}{c}{Algorithm~\ref{alg}} & \phantom{abc} & \multicolumn{3}{c}{DSPG}\\ 
\cmidrule{2-4} \cmidrule{6-8} $n$ & Iterations & Time (s) & Gap &&  Iterations & Time (s) & Gap \\
$50$ & 85 & 0.10 & 7.12e-9 && 4199 & 241.21 & 3.24e-5 \\ 
$75$ & 127 & 0.29 & 1.17e-8 && 5000 & 1420.47 & 8.55e-5\\
$100$ & 180 & 0.64 & 2.54e-8  && 5000 & 4707.48 & 2.81e-3\\
    \hline
	\multicolumn{8}{c}{$p = \left \lfloor{0.3*n}\right \rfloor\ $ (constrained case)} \\
	\hline		
		&\multicolumn{3}{c}{Algorithm~\ref{alg}} & \phantom{abc} & \multicolumn{3}{c}{DSPG}\\  
\cmidrule{2-4} \cmidrule{6-8} $n$ & Iterations & Time (s) & Gap &&  Iterations & Time (s) & Gap \\
$50$ & 170 & 0.28 & 1.60e-8 && 5000 & 290.35 & 3.02e-6 \\ 
$75$ & 149 & 0.32 & 2.93e-8 && 5000 & 1415.39 & 5.06e-4\\
$100$ & 220 & 0.91 & 3.81e-8  && 5000 & 4361.37 & 4.25e-3\\
\bottomrule 
\end{tabular}

\end{table}
The result in Table~\ref{table:comparesyn-mid} indicates that the proposed method is faster than DSPG, and it also outputs solutions with higher accuracy. In addition, DSPG takes more iterations for the convergence 
as shown in Figure~\ref{fig:graph-compare-n100-2}.
In particular, the projection onto $\Z$ in DSPG does not capture the structure 
of $\Q^\top$, therefore, the projection in DSPG is not effective compared to 
the projection discussed in Section~\ref{subproblem}.

\begin{figure}[tp]
    \centering
    \includegraphics[clip, trim=1.5cm 9cm 1.5cm 9cm,scale = 0.6]{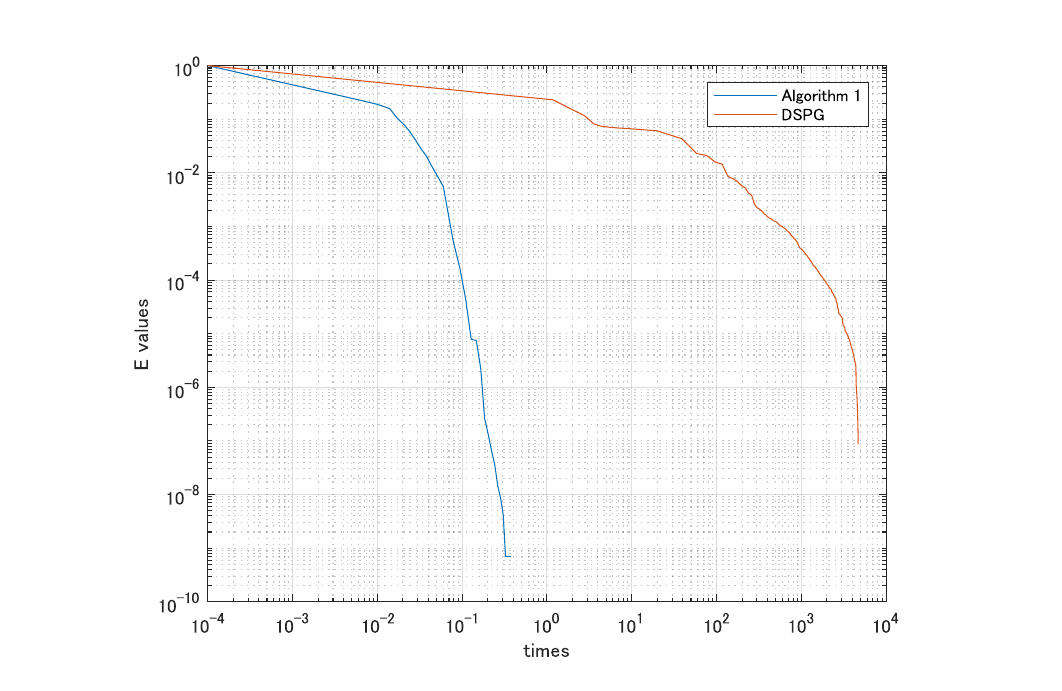}
    \caption{Comparison of the convergence rate for randomly generated synthetic data with $n = 100$ and $p = 2$}
    \label{fig:graph-compare-n100-2}
\end{figure}

Furthermore,
only Algorithm~\ref{alg} solves large instances with $n \ge 500$ in 5000 iterations, 
as indicated in Table~\ref{table:comparesyn-high}.
In the case of $n = 4000$, Algorithm~\ref{alg} can solve the unconstrained problem in 1455 seconds, and constrained problems in 4106 seconds.
\begin{table}[H]\centering \ra{1.3}
	\caption{Performance of Algorithm~\ref{alg} for large instances}
	\label{table:comparesyn-high} 
	\addtolength{\tabcolsep}{-2pt}
\scalebox{0.9}{
	\begin{tabular}{@{}rrrrrrrrrrrr@{}}\toprule &\multicolumn{3}{c}{$p=0$ (unconstrained)} & \phantom{abc} &\multicolumn{3}{c}{$p=2$ (constrained)} & \phantom{abc} &\multicolumn{3}{c}{ $p = \left \lfloor{0.3*n}\right \rfloor\ $ (constrained)}\\ 
			\cmidrule{2-4} \cmidrule{6-8} \cmidrule{10-12}  \multicolumn{1}{c}{$n$} & Iterations & Time (s) & Gap && Iterations & Time (s) & Gap && Iterations & Time (s) & Gap\\
	500 & 146 & 15.18 & 3.88e-9 && 241 & 21.73 & 1.89e-8 && 276 & 28.69 & 3.15e-8\\ 
	1000 & 113 &  54.92 & 4.31e-9 && 727 & 322.27 & 1.34e-8 && 311 & 155.72 & 1.39e-8\\
	2000 & 89 & 257.94 & 1.26e-9  && 411 & 1082.65 & 8.42e-9 && 234 & 697.72 & 1.31e-8\\
	4000 & 77 &  1455.74 & 5.23e-9 && 221 & 4106.26 & 1.42e-8 && 207 & 4120.75 & 1.65e-8\\
	\bottomrule 
	\end{tabular}}
	\end{table}




\subsection{Clustering Structure Covariance Selection}\label{subsec:covselection}

In this experiment, we generated input data following the procedure in 
Lin et al.~\cite{lin2020estimation}. We first generated a matrix 
in \cite{lin2020estimation}. 
We compute and construct the covariance matrix $\C \in \SMAT^n$ from $10n$ samples of the multivariate Gaussian distribution $\NC(\0, \BSigma)$, and  $n_G$ is the number of clusters of coordinates.
For the constraints, we set $\Omega_p$ with $p = \left \lfloor{0.3*n}\right \rfloor$. We employ the weight parameters
$\lambda = \frac{\rho}{N}$, where $N = \frac{n(n-1)}{2}$.


Similarly to Section~\ref{sec:syntheticexp}, we divide the experiments into two parts. 
The first part is the comparison between the three algorithms, 
and the second part is for
Algorithm~\ref{alg} with large matrices, where
we increase the matrix size $n$ and the number of clusters $n_G$.
The parameter of $\rho$ is $0.001$ in the first part, while $\rho$ is adjusted 
in each experiment in the second part
to balance the sparsity and the clustering structure
from the relative error and F-score obtained 
in preliminary experiments.

Table~\ref{table:comparecov-low} shows numerical results on the problem 
\eqref{primalexp} with clustering structure covariance selection with small matrices. 
The results in the table show that the Algorithm~\ref{alg} is still 
advantageous. It requires much less computational time than Logdet-PPA and has a 
better convergence speed than DSPG. The result of the case when $(n,n_G) = (25,5)$ 
shows that Algorithm~\ref{alg} can solve the problem in 0.18 
seconds, while DSPG and Logdet-PPA take 7.07 seconds and 13.82 seconds, respectively. 
From the results on medium matrices in Table~\ref{table:comparecov-mid}, 
Algorithm~\ref{alg} again attains a better convergence rate than 
DSPG.
Table~\ref{table:comparecov-high} shows the results on large instances. Algorithm~\ref{alg} is effective even for large problems. It can solve problems in a reasonable time and can output highly accurate solutions. It can satisfy the stopping criterion of $10^{-16}$ for the problem with a large matrix size $n = 4000$ in 14 minutes and 24 seconds. 


\begin{table}[H] \centering \ra{1.3}
\caption{Comparison using covariance selection data with small matrices}
\addtolength{\tabcolsep}{-2pt}
\scalebox{0.9}{
\begin{tabular}{@{}rrrrrrrrrrrr@{}}\toprule &\multicolumn{3}{c}{Algorithm~\ref{alg}} & \phantom{abc}& \multicolumn{3}{c}{DSPG} & \phantom{abc} & \multicolumn{3}{c}{Logdet-PPA}\\ 
\cmidrule{2-4} \cmidrule{6-8} \cmidrule{10-12} $(n , n_G)$ & Iterations & Time (s) & Gap &&  Iterations & Time (s) & Gap &&  Iterations & Time (s) & Gap \\
$(10,5)$ & 57 & 0.01 & 4.70e-11 && 132 & 0.02 & 2.85e-11 && 29 & 0.34 & 2.38e-7 \\ 
$(20,5)$ & 90 & 0.04 & 2.86e-10 && 626 & 0.55 & 1.49e-9 && 67 & 2.44 & 2.93e-7\\
$(25,5)$ & 467 & 0.18 & 4.55e-9 && 3435 & 7.07 & 5.31e-8 && 89 & 13.82 & 1.23e-6\\
\bottomrule 
\end{tabular}}
\label{table:comparecov-low} 
\end{table}
\begin{table}[H] \centering \ra{1.3}
\caption{Comparison using covariance selection data with medium matrices}
\begin{tabular}{@{}rrrrrrrr@{}}\toprule		
		&\multicolumn{3}{c}{Algorithm~\ref{alg}} & \phantom{abc} & \multicolumn{3}{c}{DSPG}\\ 
\cmidrule{2-4} \cmidrule{6-8} $(n , n_G)$ & Iterations & Time (s) & Gap &&  Iterations & Time (s) & Gap \\
$(50,5)$ & 146 & 0.22 & 1.88e-9  && 2518 & 153.35 & 1.39e-7 \\ 
$(75,5)$ & 908 & 2.74 & 1.94e-7 && 5000 & 1541.50 & 1.89e-4\\
$(100,5)$ & 135 & 0.67 & 2.86e-11  && 4670 & 4674.50 & 1.06e-7 \\
    \hline
\end{tabular}
\label{table:comparecov-mid} 
\end{table}
\begin{table}[H] \centering \ra{1.3}
\caption{The performance of Algorithm~\ref{alg} for covariance selection data in large matrices}
\begin{tabular}{@{}cccc@{}}\toprule &\multicolumn{3}{c}{Algorithm~\ref{alg}} \\ 
\cmidrule{2-4} $(n , n_G , \rho)$ & Iterations & Time(s) & Gap\\
$(500,10,10^{-3})$ & 168 & 17.18 & 5.72e-10\\ 
$(1000,20,10^{-4})$ & 87 & 42.28 & 2.65e-11\\
$(2000,50,10^{-4})$ & 73 & 195.78 & 4.85e-11\\
$(4000,50,10^{-5})$ & 61 & 864.55 & 1.09e-11\\ 
\bottomrule 
\end{tabular}
\label{table:comparecov-high} 
\end{table}

\subsection{Real Data Experiments}\label{sec:realdata}


We also executed experiments on a real animal dataset in \cite{kemp2008discovery}. 
The dataset consists of binary values which are the answers to $d = 102$ true-false questions on $n = 33$ animals. We followed the procedure used in~\cite{egilmez2017graph,lin2020estimation}, computed the input matrix $\C = \overline{\bm S} + \frac{1}{3}\I$, where $\overline{\bm S}$ is the sample covariance matrix and $\I$ is the identity matrix. We applied the model \eqref{primalexp} with no constraints and took $\rho = 0.01$ and $\lambda = \frac{4 \rho}{n(n-1)}$.

Table~\ref{table:compareanimal} summarizes the result of 
Algorithm~\ref{alg} and DSPG on the animal dataset problem, and Figure~\ref{fig:graph-compare-animal} show their convergence rates.
We did not include Logdet-PPA here due to being out of memory. 
From these results, we can also see that 
Algorithm~\ref{alg} can obtain a highly accurate 
the solution in a short computation time.

\begin{table}[H]\centering \ra{1.3}
\caption{Comparison of the performance for animal dataset}
\begin{tabular}{@{}rrrrrrrr@{}}\toprule		
		&\multicolumn{3}{c}{Algorithm~\ref{alg}} & \phantom{abc} & \multicolumn{3}{c}{DSPG}\\ 
\cmidrule{2-4} \cmidrule{6-8} $n$ & Iterations & Time (s) & Gap &&  Iterations & Time (s) & Gap \\
$33$ & 29 & 0.03 & 2.50e-11  && 497 & 5.03 & 5.44e-9\\
    \hline
\end{tabular}
\label{table:compareanimal} 
\end{table}
\begin{figure}[H]
	\centering
	\includegraphics[clip, trim=1.5cm 9cm 1.5cm 9cm,scale = 0.6]{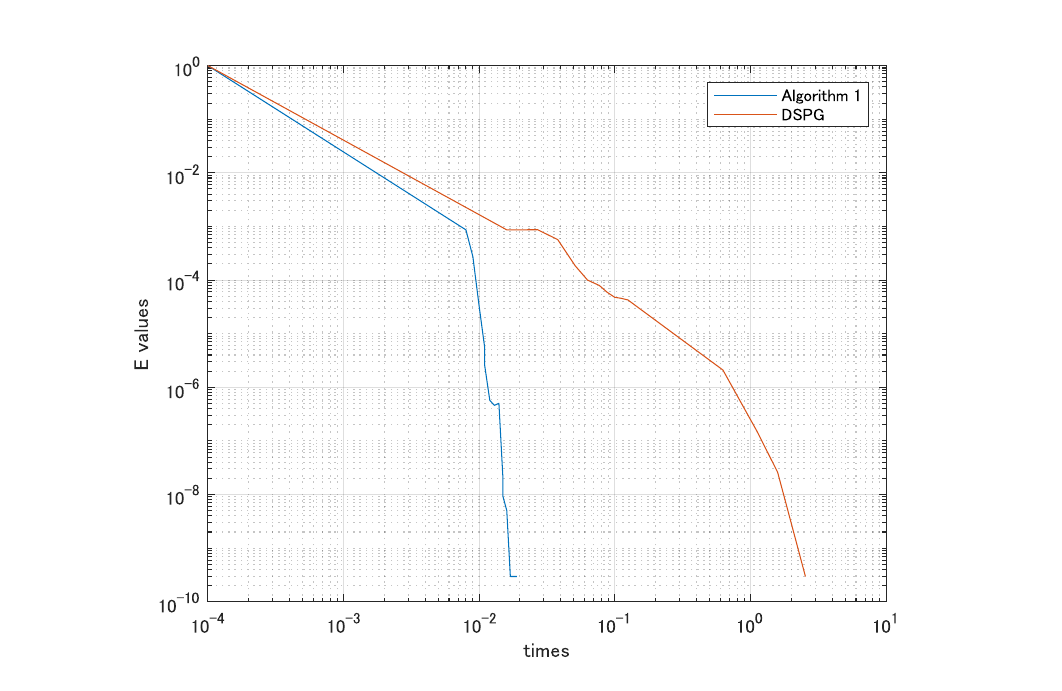}
	\caption{Comparison of the convergence rate for Animal Dataset}
	\label{fig:graph-compare-animal}
\end{figure}

\section{Conclusion}\label{sec:conclusion}

In this paper, we extended the dual spectral projected gradient method in \cite{nakagaki2020dual} to
the novel DSPG method (Algorithm~\ref{alg})
for solving \eqref{dual_pre}.
To reduce the length of the gradient vector of the dual objective function, we replaced $\Q^\top (\bm Z)$ with the new variable matrix $\bm S$,
and we developed an efficient method to compute the projection onto $\SC$.
We established the convergence
of Algorithm~\ref{alg} to the optimal value.
We also showed that the projection in the proposed method can be computed in $\OC(n^2 \log n)$ operations using the pool-adjacent-violators algorithm.
The results from numerical experiments on randomly generated synthetic data, covariance selection, and animal data indicate that the proposed method obtains accurate solutions in a short computation time and solves large instances.

One of the future directions of our research is an extension of the DSPG algorithm for 
solving more general types of log-determinant semidefinite programming, for example, 
sparse and locally constant Gaussian graphical models proposed by 
Honorio et al.~\cite{honorio2009sparse}. 

\section*{Data Availability}
The test instances in Sections 5.1 and 5.2 were generated randomly following the steps described in these sections.
The test instance in Section 5.3 was generated based on the dataset in \cite{kemp2008discovery}.

\section*{Conflict of Interest}
All authors have no conflicts of interest.

\section*{Acknowledgments}
The research of M.~Y.~was partially supported by JSPS KAKENHI (Grant Number: 21K11767).

\begin{small}
\bibliography{mybib.bib}
\end{small}

\end{document}